\documentclass[11pt,reqno]{amsart}
\usepackage{graphicx}
\usepackage{color}
\usepackage{amsmath,amssymb,amsthm,amsfonts}
\usepackage{graphicx}
\usepackage{color}
\usepackage{multicol}
\def\R{\mathbb{R}}
\makeatletter

\newcommand{\Rmnum}[1]{\expandafter\@slowromancap\romannumeral #1@}
\makeatother
\newcommand{\D}{\displaystyle}
\newtheorem{thm}{Theorem}[section]
\newcommand{\norm}[1]{\lVert#1\rVert}
\newtheorem{definition}{Definition}[section]

\newtheorem{lemma}[thm]{Lemma}
\newtheorem{remark}{Remark}[section]
\newtheorem{theorem}[thm]{Theorem}
\newtheorem{proposition}[thm]{Proposition}

\newcommand{\abs}[1]{\left\vert#1\right\vert}
\usepackage[numbers,sort&compress]{natbib}
\begin{document}

\author{Hongyun Peng}
\address{School of Applied Mathematics, Guangdong University of Technology, Guangzhou, 510006,
China} \email{hypengmath@gdut.edu.cn}

\author{Zhi-An Wang}
\address{Department of Applied Mathematics, Hong Kong Polytechnic University,
Hung Hom, Kowloon, Hong Kong}
\email{mawza@polyu.edu.hk}

\author{Changjiang Zhu}
\address{School of Mathematics, South China University of Technology, Guangzhou,
510641, China} \email{machjzhu@scut.edu.cn}

\title[Global weak solutions and asymptotics  of a singular PDE-ODE chemotaxis System]{Global weak solutions and asymptotics  of a singular PDE-ODE chemotaxis system with discontinuous data}

\begin{abstract}
{This paper is concerned with the well-posedness and large-time behavior of a two dimensional PDE-ODE hybrid chemotaxis system describing the initiation of tumor angiogenesis. We first transform the system via a Cole-Hopf type transformation into a parabolic-hyperbolic system and show that the solution of the transformed system converges to a constant equilibrium state as time tends to infinity. Finally we reverse the Cole-Hopf transformation and obtain the relevant results for the pre-transformed PDE-ODE hybrid system. In contrast to the existing related results, where continuous initial datum is imposed, we are able to prove the asymptotic stability for discontinuous initial data with large oscillations. The key ingredient in our proof is the use of so-called ``effective viscous flux", which enables us to obtain the desired energy estimates and regularity. The technique of ``effective viscous flux" turns out to be a very useful tool to study chemotaxis systems with initial data having low regularity and was rarely (if not) used in the literature for chemotaxis models. }
\end{abstract}

\subjclass[2000]{35A01, 35B40, 35Q92, 92C17}

\keywords{Chemotaxis, asymptotic stability,
discontinuous initial data, effective viscous flux}

\maketitle

\numberwithin{equation}{section}
\bigbreak
\section{Introduction}
In this paper we consider the following PDE-ODE hybrid chemotaxis model
\begin{eqnarray}\label{ks1}
\left\{\begin{array}{lll}
u_t=\Delta u-  \nabla\cdot(\xi u \nabla\ln c),\\
c_t=- \mu u c
\end{array}\right.
\end{eqnarray}
which was proposed in \cite{LSN} to describe the interaction between signaling molecules vascular endothelial growth factor (VEGF) and vascular endothelial cells during the initiation of tumor angiogenesis (see also in \cite{CPZ1, CPZ3}), where $u(x,t)$ and $c(x,t)$ denote the density of vascular endothelial cells and concentration of VEGF, respectively. The parameter $\xi>0$ is referred to as the chemotactic coefficient measuring the strength
of chemotaxis and $\mu$ denotes the degradation rate of the chemical VEGF. Here the chemical diffusion is neglected since it is far less important than its interaction with endothelial cells (cf. \cite{LSN}). On the other hand, the system \eqref{ks1} can also be regarded as a particular form of the well-known Keller-Segel system proposed in the seminal paper \cite{KS} describing the propagation of traveling wave band formed by bacterial chemotaxis observed in the experiment of Adler \cite{Adler}.

The distinguishing feature of the model (\ref{ks1}) is twofold: (i) the first equation of (\ref{ks1}) contains a logarithmic sensitivity function $\ln c$ which is singular at ${ c=0}$ but is a very meaningful sensitivity representation (cf. \cite{Othmer97, Kalinin};  (ii) the second equation of (\ref{ks1}) is an ODE and hence lacks a spatial structure.  Either of these features is the source of challenges for mathematical analysis and numerical computations. To overcome these obstacles,  a Cole-Hopf type transformation
 \begin{equation}\label{ch}
{\bf v}=-\frac{1}{\mu}\nabla(\ln c) =-\frac{1}{\mu}\frac{\nabla c}{c}
 \end{equation}
has been introduced in \cite{Levine97, Wang08}, which transforms the system (\ref{ks1}) into a parabolic-hyperbolic system:
\begin{equation}\label{hp}
\begin{cases}
u_t-\chi\nabla\cdot(u{\bf v})=\Delta u,\\
{\bf v}_t-\nabla u=0,
\end{cases}
\end{equation}
where $\chi=\mu\xi>0$, ${\bf v}$ is a gradient vector and curl${\bf v}=0$.  Apparently the transformed system (\ref{hp}) is much more manipulable mathematically than the original system (\ref{ks1}) since the singularity vanishes and the quantity ${\bf v}$ possesses a spatial structure through $u$.  In this paper, we
shall consider the well-posedness and asymptotic behavior of solutions to (\ref{hp}) with initial data
\begin{equation}\label{Initial} (u, {\bf v})(x, 0)=(u_0, {\bf v}_0)(x), \ \ x\in
{\mathbb{R}}^2
\end{equation}
subject to the following asymptotic states
\begin{equation}\label{Boundary}
 (u, {\bf v})(\pm\infty, t)=(\bar{u}, {\bf0}),
\end{equation}
where $(\bar{u}, {\bf0})$  with $\bar{u}>0$ is a constant ground state of \eqref{hp}). Then we reverse the Cole-Hopf transformation \eqref{ch} and get the well-posedness and asymptotic behavior of solutions to the original PDE-ODE chemotaxis system (\ref{ks1}).

There has been a considerable amount of interesting works carried out for the transformed system (\ref{hp}). The one-dimensional problem has been first studied extensively from various aspects such as the traveling wave solutions \cite{jin13, Li09, Li10, Lij13, Mei-peng-wang, PW2019, Wang-review}, global dynamics of solutions in the whole space \cite{GXZZ, Li-pan-zhao, zhang-tan-sun, Martinez} or in the bounded interval \cite{Hou, Li-Zhao-JDE, Li112, zhang07}. For the multidimensional whole space $\R^2$, the nonlinear stability of planar traveling wave was recently established in \cite{Chae2}. When the initial datum is close to the constant ground state $(\bar{u}, {\bf 0})$, numerous results have been  obtained to the system (\ref{hp}).  First the blowup criteria of  classical solutions was established in \cite{Fan-zhao, Li111} and the long-time behavior of solutions was obtained in \cite{Li111} if $(u_0-\bar{u} , {\bf v_0})\in H^s(\R^d)$ for $s>\frac d2+1$ and $\norm{(u_0-\bar{u} , {\bf v_0})}_{H^s\times H^s}$ is small. Later, Hao \cite{Hao} established the global existence of mild solutions in the critical Besov space $\dot{B}_{2,1}^{-\frac12}\times (\dot{B}_{2,1}^{\frac12})^d $ with minimal regularity in the Chemin-Lerner space framework. The global well-posedness of strong solutions of \eqref{hp} in $\R^3$ was established in \cite{DL} if $\norm{(u_0-\bar{u} , {\bf v_0})}_{L^2\times H^1}$ is small. If the initial datum has a higher regularity such that $\norm{(u_0-\bar{u} , {\bf v_0})}_{H^2\times H^1}$ is small, the algebraic decay of solutions was further derived in \cite{DL}. Wang, Xiang and Yu \cite{Wang-xiang-yu} established the global existence and time decay rates of solutions of \eqref{hp} in $\R^d$ for $d=2,3$ if $(u_0-\bar{u} , {\bf v_0})\in H^2(\R^d)$ and $\norm{(u_0-\bar{u} , {\bf v_0})}_{H^1\times H^1}$ is small. Recently the small energy solution in $\R^d (d=2,3)$ was established in \cite{WWZ}. In the multidimensional bounded domain $\Omega \subset \R^d(d=2,3)$, the global existence and exponential decay rates of solutions under Neumann boundary conditions were obtained in \cite{Li112,RWWZZ} for small data, and local existence of classical solutions in two dimensions with Dirichlet boundary conditions was given in \cite{HWJMPA}. When the Laplacian (diffusion) in \eqref{hp} was modified to a fractional Laplacian, the global existence of solutions of  \eqref{hp} in a torus with periodic boundary conditions in some dissipation regimes was established in \cite{Rafael1,Rafael2}.

In the above-mentioned results, all solutions obtained for  \eqref{hp} are classical thanks to the high regularity and smallness of the $H^s(s\geq 1)$-norm on the initial data. The goal of this paper is to exploit the parabolic-hyperbolic system \eqref{hp} with rougher initial data which are allowed to be discontinuous and possess large amplitude. Specifically we consider the initial data $(u_0, v_0)$ in some appropriate $L^p$ space and allow $\norm{u_0-\bar{u}}_{L^\infty}+\norm{\mathbf{v}_0}_{L^\infty}$ to be arbitrarily large.
The low regularity of initial data and large amplitude bring us many new difficulties in analysis compared to the existing works (cf. \cite{DL, Li111, Wang-xiang-yu}). In particular, the integrability of $\nabla {\bf v}$, which plays a crucial role in the analysis of \cite{DL, Li111, Wang-xiang-yu}, can be easily attained therein but appears to be unattainable for our less regular initial data due to the hyperbolic effect of the second equation.  Hence one can only expect to examine $L^{p}$-integrability for ${\bf v}$ and hence obtain weak solutions instead of strong (classical) solutions.  The key idea of gaining  $L^{p}$-integrability of ${\bf v}$ developed in this paper is to rewrite the first equation of (\ref{hp}) as $\tilde{u}_t=\nabla \cdot \mathbf{F}$ inspired by the works \cite{Hoff-1995, Lions2} on the Navier-Stokes equations, where $\mathbf{F}$ is the so called ``effective viscous flux" defined by
\begin{equation}\label{flux}
\mathbf{F}=\nabla u+\chi u\bf{v}.
\end{equation}
Then we rewrite the second equation of (\ref{hp}) as
$${\bf v}_t=-\chi u {\bf v} +{\bf F}$$
and try to attain the regularity of ${\bf v}$ through ${\bf F}$ whose regularity is easier than ${\bf v}$ to obtain. We remark that the dynamics of PDEs with discontinuous data has been an important topic arising from fluid mechanics and gas dynamics in order to understand how the discontinuity evolves in the fluid. For this, Hoff with his collaborators have developed some nice ideas and  obtained  a series of important results in this aspect (cf. \cite{Hoff-jde, Hoff-1995, Hoff-arma, Hoff-cpam, Hoff-santos}). We refer to \cite{Chen-xu-zhang, Hu-lin, KLLZ-JDE, MASMOUDI} and references therein for further development on the discontinuous data problem. However in our current work the ``effective viscous flux" ${\bf F}$ exists in its divergence form (i.e. $\nabla\cdot {\bf F})$, which is different from that used in the Navier-Stokes equations where the ``effective viscous flux" exists in its gradient form and $L^{p}$-norm of $\nabla {\bf F}$ can be achieved directly (cf. \cite{Hoff-1995, Hoff-arma, Hoff-cpam, Hoff-santos, Lions2}). This implies that besides the estimate of $\nabla\cdot {\bf F}$, we have to estimate for the curl of ${\bf F}$ (denoted by $\nabla^{\bot}{\bf F}$) in order to get the estimates for $\nabla{\bf F}$ and hence the regularity for ${\bf F}$, which is a key difference from the Navier-Stokes equation.  The ``effective viscous flux" technique appears to be a very powerful tool to study chemotaxis systems with low regular initial data and was rarely (if not) used in the literature.


%
%
Since the dependence of solutions on $\chi$ and $\bar{u}$ is not the interest of this paper, we henceforth assume $\chi=\bar{u}=1$ throughout the paper for simplicity without further clarification.

To state our results, we first introduce the definition of weak solutions of \eqref{hp}-\eqref{Boundary}.

\begin{definition}[Weak solutions]
\label{pw7-defi} We say that $(u, {\bf v})$ is a weak solution of \eqref{hp}-\eqref{Boundary}, if $(u,{\bf v})$ is suitably
integrable and satisfies for all test functions $\Psi\in C_0^\infty(\mathbb{R}^2\times [0, \infty))$ that
\begin{equation*}
\begin{split}
\int_{\mathbb{R}^2}u_0\Psi_0dx+\int_0^\infty\int_{\mathbb{R}^2}\left(u\Psi_t-\nabla u\cdot\nabla\Psi\right)dxdt=\int_0^\infty\int_{\mathbb{R}^2}u{\bf v}\cdot\nabla\Psi dxdt
\end{split}
\end{equation*}
and
\begin{equation*}
\begin{split}
\int_{\mathbb{R}^2}{\bf v}_0^j\Psi_0dx+\int_0^\infty\int_{\mathbb{R}^2}\left({\bf v}^j\Psi_t-u\Psi_{x_j}\right)dxdt=0,
\end{split}
\end{equation*}
where $j=1,2$ and $\Psi_0(x)=\Psi(x,0)$.
\end{definition}
Then our main results in this paper are given as follows:
\begin{theorem}\label{pw7-th}
Let $4<p_0<\infty$ and the initial
data satisfy
\begin{equation}\label{pw6-ict1}
u_0-1 \in L^2(\mathbb{R}^2), \ \ {\bf v}_0 \in L^2(\mathbb{R}^2)\cap L^{p_0}(\mathbb{R}^2),\ \  u_0\geq0, \ \ \nabla^{\bot}\cdot {\bf v}_0=0
\end{equation}
where $\nabla^{\bot}=(\partial_2, -\partial_1)$ denotes the curl operator. Then for any constant $M>0$ with $\|{\bf v}_0\|_{L^{p_0}(\R^2)}\le M$, there exists a constant $\varepsilon>0$ depending on $M$, such that if
\begin{equation*}
\norm{u_0-1}_{L^2(\R^2)}^2+\norm{\mathbf{v_0}}_{L^2(\R^2)}^2=\theta_0\leq \varepsilon,
\end{equation*}
then the Cauchy problem \eqref{hp}-\eqref{Boundary} has a
 global weak solution $(u,{\bf v})(x,t)$ satisfying
\begin{equation}\label{excon}
\begin{cases}
\begin{split}
&u-1 \in  L^\infty([0,\infty); L^2(\mathbb{R}^2))\cap C((0,\infty);C(\mathbb{R}^2)), \ \nabla u\in L^2([0,\infty);L^2(\mathbb{R}^2)),\\
&{\bf v} \in L^\infty([0,\infty); L^2(\mathbb{R}^2)\cap L^{p_0}(\mathbb{R}^2))\cap C([0,\infty), H^{-1}(\mathbb{R}^2)).
\end{split}
\end{cases}
\end{equation}
and the following asymptotic convergence
\begin{equation}\label{pw7-linff}
\begin{split}
\|u(x,t)-1\|_{L^{p_1}(\R^2)}\to
0, \  \
\|{\bf v}(x,t)\|_{L^{p_2}(\R^2)}\to
0~~as~~t\to\infty,
\end{split}
\end{equation}
where $2< p_1\le\infty$ and $2< p_2<p_0<\infty$.
\end{theorem}
\begin{remark}\em{
The above results hold true regardless of the strength of initial perturbation from the constant ground state $(1,0)$, namely the amplitude $\norm{u_0-1}_{L^\infty}+\norm{\bf{v}_0}_{L^\infty}$ can be arbitrarily large. The initial conditions \eqref{pw6-ict1} imply that the initial data of the Cauchy problem \eqref{hp}-\eqref{Boundary} could be discontinuous (such as piecewise constant with arbitrarily large jump discontinuities), which brings great challenges to the analysis. Moreover the condition  $\nabla^{\bot}\cdot {\bf v}_0=0$ is a natural consequence the Cole-Hopf transformation \eqref{ch}.
}
\end{remark}

To prove Theorem \ref{pw7-th}, we first mollify the initial data with a mollifying parameter $\delta$ and obtain the global smooth solution $(u^\delta, \bf{v}^\delta)$ depending on the mollifying parameter $\delta$. Then we pass to the limit as $\delta\to 0$ and obtain a weak solution of \eqref{hp}-\eqref{Boundary}.  The core of the proof is to derive the global {\it a priori} estimates independent of the mollifying parameter $\delta$. In this connection, the approaches and estimates developed in previous works \cite{Li111, Wang-xiang-yu} for small-amplitude and continuous initial data are not adequate for our current problem. In this paper, we shall introduce the so called ``effective viscous flux" technique and make a full use of the structure
of \eqref{hp} to obtain the desired uniform-in-$\delta$ estimates.

Converting the results from $v$ to $c$ by reversing the Cole-Hopf transformation (\ref{ch}), we get the results for the original chemotaxis model \eqref{ks1}.

\begin{theorem}\label{th-dlor}
Let $4<p_0<\infty$ and the initial
data satisfy
\begin{equation*}
u_0-1 \in L^2(\mathbb{R}^2), \ \ u_0(x)\geq0, \ \ \nabla\ln c_0\in L^2(\mathbb{R}^2)\cap L^{p_0}(\mathbb{R}^2)
\end{equation*}
Then for any constant $M>0$ with $\|\nabla\ln c_0\|_{L^{p_0}(\R^2)}\le M$, there exists a constant $\varepsilon>0$ depending on $M$, such that if
\begin{equation*}
\norm{u_0-1}_{L^2(\R^2)}^2+\norm{\nabla\ln c_0}_{L^2(\R^2)}^2 \leq \varepsilon,
\end{equation*}
the Cauchy
problem \eqref{hp}-\eqref{Boundary} has a
 global weak solution $(u,c)(x,t)$ satisfying
 \begin{equation*}
\begin{cases}
\begin{split}
u-1 &\in  L^\infty([0,\infty); L^2(\mathbb{R}^2)), \ \ \ \   \nabla u\in L^2([0,\infty);L^2(\mathbb{R}^2)),\\
u-1 &\in C((0,\infty);C(\mathbb{R}^2)), \ \ \ {\nabla\ln c} \in C([0,\infty), H^{-1}(\mathbb{R}^2)),\\
{\nabla\ln c} &\in L^\infty([0,\infty); L^2(\mathbb{R}^2)\cap L^{p_0}(\mathbb{R}^2))\cap L^4([0,\infty); L^4(\mathbb{R}^2))
\end{split}
\end{cases}
\end{equation*}
and
\begin{equation*}
\begin{split}
\|u-1\|_{L^{p_1}(\mathbb{R}^2)}\to
0, \  \
\|\nabla\ln c\|_{L^{p_2}(\mathbb{R}^2)}\to
0~~as~~t\to\infty,
\end{split}
\end{equation*}where $2< p_1\le\infty$ and $2< p_2<p_0<\infty$.

Furthermore, if $c_0\in L^\infty(\mathbb{R}^2), \ c_0>0$, then there exists a positive constant $C$ independent of $t$ such that the solution has the following decay rates in time:
\begin{equation}\label{cdecay}
\begin{split}
\norm{c}_{L^\infty(\mathbb{R}^2)}\le  Ce^{-\frac34 t}.
\end{split}
\end{equation}
\end{theorem}

The rest of paper is organized as follows. In Section 2, we collect some
elementary facts and inequalities that will be used in later analysis. Section 3 is devoted to deriving the {\it a priori} estimates on smooth solutions. Then we prove Theorem \ref{pw7-th} in Section 4. Finally we prove Theorem \ref{th-dlor} in section 5.

\section{Preliminaries}
In this section, we will recall and prove some basic results that will be used later. Before embarking on this, we first introduce some notations used throughout this paper.

{\bf Notations}.

\begin{itemize}
\item In what follows, $C$ denotes a generic positive constant which may vary in the context;
\item $H^k(\mathbb{{R}}^2)$ denotes the usual $k$-th order Sobolev space on $\mathbb{{R}}^2$ with norm $$\|f\|_{H^k(\mathbb{R}^2)}:=\Big(\sum_{j=0}^{k}\int_{\R^2}
|\partial_x^jf|^2dx\Big)^{1/2}.$$
For simplicity, we denote $\|\cdot\|:=\|\cdot\|_{L^2(\mathbb{R}^2)}$,
$\|\cdot\|_k:=\|\cdot\|_{H^k(\mathbb{R}^2)}$, and furthermore $L^2(\R^2)$ will be abbreviated as $L^2$ without confusion.
\item  We denote the curl operator as
\begin{equation}\label{pw7-dev}
\nabla^{\bot}=(\partial_2, -\partial_1),
\end{equation}
\item For simplicity we set
\begin{equation}\label{pw7-cc}
\theta_0=\norm{u_0-1}^2+\norm{\bf{v}_0}^2.
\end{equation}
Since $\theta_0$ will be assumed to be small in this paper, we assume that $\theta_0<1$ without loss of generality in the sequel.
\end{itemize}

We start with the local existence and blowup criterion of smooth solutions to the Cauchy problem \eqref{hp}-\eqref{Boundary} established in
\cite{Li111,Fan-zhao}.
\begin{lemma}[\cite{Li111}]\label{pw7-local}
 Let $s>\frac d2+1$ and $(u_0-1, {\bf v_0})\in H^s(\R^d)$. Then there
exists a small-time $T_*=T_*(\|u_0-1\|_{H^s}, \|{\bf v_0}\|_{H^s(\R^d)})>0$ such that the Cauchy problem \eqref{hp}-\eqref{Boundary} a unique solution $(u-1, {\bf v})\in L^\infty((0,T_*], H^s(\R^d))$.
\end{lemma}

\begin{lemma}[\cite{Fan-zhao}]
 Let $s>\frac d2+1$ and $(u_0-1, {\bf v_0})\in H^s(\R^d)$. Let $(u, {\bf v})$ be the unique local solution in Lemma \ref{pw7-local} with the maximal
lifespan $T_*>0$. If
\begin{equation}\label{pw7-blowup}
\begin{split}
\int_0^{T_*}\norm{{\bf v}}_{L^q(\R^d)}^{\frac{2q}{q-d}}<\infty, \ with \ d<q\le \infty,
\end{split}
\end{equation}
then the solution can be extended beyond $T_*>0$.
 \end{lemma}

Next, we introduce a change of variable $\tilde{u}=u-1$.
Thus, the problem \eqref{hp}-\eqref{Initial} turns
into
\begin{equation}\label{pw7-hptr1}
\begin{cases}
\tilde{u}_t-\Delta\tilde{u}=\nabla\cdot(\tilde{u}{\bf v})+\nabla\cdot {\bf v},\\
{\bf v}_t-\nabla\tilde{u}=0,\\
(\tilde{u}, {\bf v})(x,0)=(u_0-1, {\bf v}_0)(x).
\end{cases}
\end{equation}
Then the ``effective viscous flux'' ${\bf F}$ defined in \eqref{flux} can be written as
\begin{equation}\label{pw7-evf}
{\bf F}=\nabla\tilde{u}+(\tilde{u}+1)\bf{v}.
\end{equation}
By the first equation of $\eqref{pw7-hptr1}$, it is easy to see that
\begin{equation}\label{pw7-fut}
\nabla\cdot {\bf F}=\tilde{u}_t.
\end{equation}

Then ${\bf F}$ has the following useful estimate.

\begin{lemma}\label{pw7-lef1}
 Let $(\tilde{u}, {\bf v})$ be a smooth solution of \eqref{pw7-hptr1}. Then there exists a
 positive constant $C$ such that
\begin{equation}\label{pw7-f2p}
\begin{split}
\norm{\nabla {\bf F}}_{L^p}
\le&C (\norm{\tilde{u}_t}_{L^p}+ \norm{\nabla^{\bot}\tilde{u}\cdot {\bf v}}_{L^p}),
\end{split}
\end{equation}
where $p>1$.
\end{lemma}
\begin{proof} First we recall the following inequality (cf. Lemma 2.4(1) in  \cite{jiu13}):
\begin{equation}\label{pw7-FLP}
\begin{split}
\norm{\nabla {\bf F}}_{L^p}\le& C\big(\norm{\nabla\cdot {\bf F}}_{L^p}+\norm{\nabla^{\bot}\cdot {\bf F}}_{L^p}\big).
\end{split}
\end{equation}
It then follows from \eqref{pw7-dev} and \eqref{pw7-evf} that
\begin{equation}\label{pw7-fcur}
\begin{split}
\nabla^{\bot}\cdot {\bf F}=&\nabla^{\bot}\cdot\nabla \tilde{u}+\nabla^{\bot}\cdot(\tilde{u}{\bf v})+\nabla^{\bot}\cdot {\bf v}\\
=&\nabla^{\bot}\cdot(\tilde{u}{\bf v})+\nabla^{\bot}\cdot {\bf v}\\
=&(\tilde{u}+1)\nabla^{\bot}\cdot {\bf v}+\nabla^{\bot}\tilde{u}\cdot{\bf v},
\end{split}
\end{equation}
where in the second equality we have used the fact that $\nabla^{\bot}\cdot\nabla=0$. Operating $\nabla^{\bot}$ on the second equation of \eqref{pw7-hptr1} and integrating the result over $(0, t)$, we have from $\nabla^{\bot}\cdot {\bf{v}}_0=0$ that
$
\nabla^{\bot}\cdot \bf{v}=0,
$
which updates \eqref{pw7-fcur} as
$
\nabla^{\bot}\cdot {\bf F}=\nabla^{\bot}\tilde{u}\cdot \bf{v}.
$
This together with \eqref{pw7-fut} and \eqref{pw7-FLP} gives
\begin{equation*}
\begin{split}
\norm{\nabla {\bf F}}_{L^p}
\le&C \norm{\tilde{u}_t}_{L^p}+ C\norm{\nabla^{\bot}\tilde{u}\cdot {\bf v}}_{L^p}.
\end{split}
\end{equation*}
Thus, the proof of Lemma \ref{pw7-lef1} is completed. \end{proof}

The following  Gagliardo-Nirenberg inequality will be frequently used in this paper.
\begin{lemma}[\cite{Friedman-A}]
Let $1\le q, r\le \infty$ and $0<a\le 1 $ such that
$$\frac{1}{p}=a\bigg(\frac 1q-\frac12\bigg)+(1-a)\frac{1}{r}.$$
Then, for any $u\in W^{1,p}(\mathbb{R}^2)\cap L^r(\mathbb{R}^2)$, there exists
a positive constant $C$ depending only on $q,
r$ and $n$ such that the following inequality holds:
\begin{equation}\label{G-Ns}
\begin{split}
\norm{u}_{L^p(\mathbb{R}^2)}\le C\norm{\nabla
u}^a_{L^q(\mathbb{R}^2)}\norm{u}^{1-a}_{L^r(\mathbb{R}^2)}.
\end{split}
\end{equation}
\end{lemma}

\section{A Priori Estimates for Approximate Solutions}
In this section, we prove Theorem \ref{pw7-th}, by
constructing approximate solutions based upon the following mollified initial data:
\begin{equation*}
\begin{split}
\tilde{u}^\delta_0=j^\delta*\tilde{u}_0, \ \ \ {\bf v}_0^\delta=j^\delta*{\bf v}_0,
\end{split}
\end{equation*}
where $\tilde{u}_0=u_0-1$ and $j^\delta$ is the standard mollifying kernel of width $\delta$ (e.g. see \cite{Adams}). Then we consider the following approximate system
\begin{equation}\label{pw7-hptr}
\begin{cases}
\tilde{u}^\delta_t-\Delta\tilde{u}^\delta=\nabla\cdot(\tilde{u}^\delta{\bf v}^\delta)+\nabla\cdot {\bf v}^\delta,\\
{\bf v}^\delta_t-\nabla\tilde{u}^\delta=0,
\end{cases}
\end{equation}
with smooth initial data $(\tilde{u}^\delta_0, {\bf v}_0^\delta)$ which satisfies
\begin{equation}\label{ini}
(\tilde{u}^\delta_0, {\bf v}^\delta_0)\in H^3(\R^2)
\end{equation}
and
\begin{equation}\label{pw8mde0}
\norm{\tilde{u}^\delta_0}^2+\norm{{\bf v}^\delta_0}^2\le\norm{\tilde{u}_0}^2+\norm{{\bf v}_0}^2=\theta_0.
\end{equation}
By Lemma \ref{pw7-local}, we can obtain the local existence of solutions to the approximate system \eqref{pw7-hptr} with initial data $(u^\delta_0, {\bf v}_0^\delta)$ satisfying (\ref{ini})-(\ref{pw8mde0}).
Next, we shall show in a sequence of lemmas that these approximate solutions
satisfy some global {\it a priori} estimates, independent of the mollifying parameter $\delta$.
These estimates will then be applied in Section 4 to obtain the solutions of
Theorem \ref{pw7-th} as the limits of these approximate solutions.

For the sake of simplicity,  we still use $(\tilde{u}, {\bf v})$ to
represent the approximate solution $(\tilde{u}^\delta, {\bf v}^\delta)$ in this section.
Let $T>0$ be a fixed time and $(\tilde{u}, {\bf v})$ be the smooth
solution to \eqref{pw7-hptr} on $\R^2\times (0, T]$. We set $\sigma=\sigma(t)=\min\{1, t\}$ and define
\begin{equation}\label{pw7-ass1}
\begin{cases}\begin{split}
&A_1(T)\triangleq\sup\limits_{t\in[0,T]}\left(\norm{\tilde{u}}^2+\norm{{\bf v}}^2\right)+\int_0^T\norm{\nabla\tilde{u}}^2 dt, \\
&A_2(T)\triangleq\sup\limits_{t\in[0,T]}\left(\sigma\norm{\nabla\tilde{u}}^2+\sigma^2\norm{\tilde{u}_t}^2+\sigma^2\norm{{\bf v}_t}^2\right)+\int_0^T\left(\sigma\norm{\tilde{u}_t}^2+\sigma^2\norm{\nabla\tilde{u}_t}^2\right)dt,\\
&A_3(T)\triangleq\sup\limits_{t\in[0,T]}\|{\bf v}\|_{L^{4}}^4+\int_{0}^T\|{\bf v}\|_{L^{4}}^4dt.
\end{split}\end{cases}
\end{equation}

Next, we shall employ the technique of {\it a priori} assumption to derive the {\it a priori} estimates for the smooth solutions of \eqref{pw7-hptr}-\eqref{pw8mde0}.
That is, we first assume that the smooth solution $(\tilde{u}, {\bf v})$ satisfies for any $t\in[0, T]$
\begin{equation}\label{pw7-assup}
\begin{split}
A_1(T)\leq 2\theta_0, \ \  A_2(T)
\le  2\theta_0^\frac12,\ \ A_3(T)\le 2\theta_0^{\eta_0}, \ \
\sup_{t\in[0,T]}\|{\bf v}\|_{L^{p_0}}\le 6M,
\end{split}
\end{equation}
where $M$ and $\theta_0$ are from Theorem \ref{pw7-th}, and $\eta_0$ is defined as
\begin{equation}\label{pw7-etas}
\begin{split}
\eta_0\triangleq \frac{p_0-4}{2(p_0-2)}\in(0,\frac12).
\end{split}
\end{equation}
Then we will derive the {\it a priori} estimates to obtain global solutions. Finally, we show the obtained global solutions satisfy the above {\it a priori} assumption and close our argument.

We depart with the $L^2$-estimate of $(\tilde{u}, {\bf v})$.

\begin{lemma}\label{pw7-th1} Let the conditions of Theorem \ref{pw7-th} hold and $(\tilde{u}, {\bf v})$ be a smooth solution of \eqref{pw7-hptr}-\eqref{pw8mde0} satisfying \eqref{pw7-assup}. Then it holds that
\begin{equation}\label{pw7-L^4}
\norm{\tilde{u}}^2+\norm{{\bf v}}^2+\int_0^T\norm{\nabla\tilde{u}}^2 dt\leq \frac{3\theta_0}{2}.
\end{equation}
\end{lemma}
\begin{proof}
\par
Multiplying the first equation of $(\ref{pw7-hptr})$ by $\tilde{u}$ and the second by $v$, adding the results
and integrating by parts over $\R^2$, we have
\begin{equation}\label{pw7-1}
\begin{split}
\frac 12\frac{d}{dt}\left(\norm{\tilde{u}}^2+\norm{{\bf v}}^2\right)+\norm{\nabla\tilde{u}}^2
 =\int_{\R^2} \nabla\cdot(\tilde{u}{\bf v})\tilde{u}dx=-\int_{\R^2}\tilde{u}{\bf v}\nabla\tilde{u}dx.
\end{split}
\end{equation}
For the term on the right-hand side of \eqref{pw7-1}, we use \eqref{G-Ns}, H\"older and Cauchy-Schwarz inequalities to estimate it as
\begin{equation*}
\begin{split}
-\int_{\R^2}\tilde{u}{\bf v}\nabla\tilde{u}dx\le& \norm{\tilde{u}{\bf v}}^2+\frac14\norm{\nabla\tilde{u}}^2\\
\le& \norm{\tilde{u}}_{L^4}^2\norm{{\bf v}}_{L^4}^2+\frac14\norm{\nabla\tilde{u}}^2\\
\le& C\norm{\tilde{u}}\norm{\nabla\tilde{u}}\norm{{\bf v}}_{L^4}^2+\frac14\norm{\nabla\tilde{u}}^2\\
\le& C\norm{\tilde{u}}^2\norm{{\bf v}}_{L^4}^4+\frac12\norm{\nabla\tilde{u}}^2,
\end{split}
\end{equation*}
which therefore updates \eqref{pw7-1} as
\begin{equation*}
\begin{split}
\frac{d}{dt}\left(\norm{\tilde{u}}^2+\norm{{\bf v}}^2\right)+\norm{\nabla\tilde{u}}^2
 \le C\norm{\tilde{u}}^2\norm{{\bf v}}_{L^4}^4.
\end{split}
\end{equation*}
Integrating the above result over $[0,t]$ and using \eqref{pw7-cc}, \eqref{pw7-ass1} and \eqref{pw7-assup}, we have
\begin{equation*}
\begin{split}
&\norm{\tilde{u}}^2+\norm{{\bf v}}^2+\int_0^T\norm{\nabla\tilde{u}}^2dt\\
 \le& \norm{\tilde{u}_0}^2+\norm{{\bf v}_0}^2+C\int_0^T\norm{\tilde{u}}^2\norm{{\bf v}}_{L^4}^4dt\\
 \le& \norm{\tilde{u}_0}^2+\norm{{\bf v}_0}^2+CA_1(T)A_3(T)\\
 \le& \theta_0+C\theta_0^{1+\eta_0}\le \frac{3\theta_0}{2},
\end{split}
\end{equation*}
provided that $C\theta_0^{\eta_0}\le \frac12$. This yields \eqref{pw7-L^4} and completes the proof of Lemma \ref{pw7-th1}.
\end{proof}

\begin{lemma}\label{pw7-leUt}
Let the conditions of Theorem \ref{pw7-th} hold and $(\tilde{u}, {\bf v})$ be a smooth solution of \eqref{pw7-hptr}-\eqref{pw8mde0} satisfying \eqref{pw7-assup}. Then it holds that
\begin{equation}\label{pw7-L^2}
\begin{split}
\sigma\norm{\nabla\tilde{u}}^2+\sigma^2\norm{\tilde{u}_t}^2+\sigma^2\norm{{\bf v}_t}^2+\int_0^T\sigma\norm{\tilde{u}_t}^2dt+\int_0^T\sigma^2\norm{\nabla\tilde{u}_t}^2dt
\le  \theta_0^\frac12,
\end{split}
\end{equation}
where  $\sigma=\sigma(t)=\min\{1,t\}$.
\end{lemma}
\begin{proof}We divide our proof into two steps.

{\bf Step 1}. We first multiply the first equation of $(\ref{pw7-hptr})$ by $\sigma u_t$ and
integrate the resulting equation over $\R^2\times[0, T]$ to get
\begin{equation}\label{pw7-es-ut}
\begin{split}
&\frac{1}{2}\sigma\norm{\nabla\tilde{u}}^2+\int_0^T\sigma\norm{\tilde{u}_t}^2dt\\
=&\frac{1}{2}\int_0^{\sigma(T)}\norm{\nabla\tilde{u}}^2dt-\int_0^T\sigma\int_{\mathbb{R}^2}
\tilde{u}{\bf v}\nabla\tilde{u}_t dxdt-\int_0^T\sigma\int_{\mathbb{R}^2}{\bf v}\nabla\tilde{u}_t dxdt,
\end{split}
\end{equation}
where we have used the fact that
\begin{equation*}
\begin{split}
\frac{1}{2}\int_0^{T}\sigma_t\norm{\nabla\tilde{u}}^2dt=\frac{1}{2}\int_0^{\sigma(T)}\norm{\nabla\tilde{u}}^2dt.
\end{split}
\end{equation*}
For the first term on the right-hand side of \eqref{pw7-es-ut}, we have from \eqref{pw7-L^4}
\begin{equation}\label{pw7-es1}
\begin{split}
\frac{1}{2}\int_0^{\sigma(T)}\norm{\nabla\tilde{u}}^2dt\le \frac{1}{2}\int_0^T\norm{\nabla\tilde{u}}^2dt \le C\theta_0.
\end{split}
\end{equation}
For the second term on the right-hand side of \eqref{pw7-es-ut}, we have from \eqref{G-Ns}, Cauchy-Schwarz inequality, \eqref{pw7-assup} and \eqref{pw7-L^4} that
\begin{equation}\label{pw7-es2}
\begin{split}
-\int_0^T\sigma\int_{\mathbb{R}^2}
\tilde{u}{\bf v}\nabla\tilde{u}_t dxdt
\le& \lambda\int_0^T\sigma^2
\|\nabla\tilde{u}_t\|^2dt+C\int_0^T
\|\tilde{u}{\bf v}\|^2dt\\
\le&\lambda\int_0^T\sigma^2
\|\nabla\tilde{u}_t\|^2dt+ C\int_0^T
\|\tilde{u}\|_{L^4}^2\|{\bf v}\|_{L^4}^2dt\\
\le&\lambda\int_0^T\sigma^2
\|\nabla\tilde{u}_t\|^2dt+ C\int_0^T
\|\tilde{u}\|_{L^2}\|\nabla\tilde{u}\|_{L^2}\|{\bf v}\|_{L^4}^2dt\\
\le&\lambda\int_0^T\sigma^2
\|\nabla\tilde{u}_t\|^2dt+ C\int_0^T
\|\nabla\tilde{u}\|^2dt+C\int_0^T
\|\tilde{u}\|^2\|{\bf v}\|_{L^4}^4dt\\
\le&\lambda\int_0^T\sigma^2
\|\nabla\tilde{u}_t\|^2dt+ C\theta_0+C\theta_0^{1+\eta_0}\\
\le&\lambda\int_0^T\sigma^2
\|\nabla\tilde{u}_t\|^2dt+ C\theta_0,
\end{split}
\end{equation}
where $\lambda>0$ is a constant to be determined later.
For the last term on the right-hand side of \eqref{pw7-es-ut}, we have
\begin{equation}\label{pw7-es23}
\begin{split}
-\int_0^T\sigma\int_{\mathbb{R}^2}
 {\bf v}\nabla\tilde{u}_{t} dxdt=&-\int_0^T\left(\sigma\int_{\mathbb{R}^2}
{\bf v}\nabla\tilde{u} dx\right)_tdt+\int_0^{\sigma(T)}\int_{\mathbb{R}^2}
{\bf v}\nabla\tilde{u} dxdt\\
&+\int_0^T\sigma\int_{\mathbb{R}^2}
{\bf v}_t\nabla\tilde{u} dxdt= I_1+I_2+I_3.
\end{split}
\end{equation}
By Cauchy-Schwarz inequality and \eqref{pw7-L^4}, we have from $0\le \sigma\le1$ that
\begin{equation*}
\begin{split}
I_1=&-\sigma\int_{\mathbb{R}^2}
{\bf v}\nabla\tilde{u} dx\le \frac{\sigma}{4}\norm{\nabla\tilde{u}}^2+\sigma\norm{{\bf v}}^2\\[2mm]
\le& \frac{\sigma}{4}\norm{\nabla\tilde{u}}^2+2\theta_0,
\end{split}
\end{equation*}

\begin{equation*}
\begin{split}
I_2=&\int_0^{\sigma(T)}\int_{\mathbb{R}^2}
{\bf v}\nabla\tilde{u} dxdt\\
\le&  \frac12\int_0^{\sigma(T)}\norm{\nabla\tilde{u}}^2dt+\frac12\int_0^{\sigma(T)}\norm{{\bf v}}^2dt\\
\le&\frac12\int_0^{T}\norm{\nabla\tilde{u}}^2dt+\frac12\sup\limits_{t\in[0,T]}\norm{{\bf v}}^2\\
\le& 2\theta_0
\end{split}
\end{equation*}
and
\begin{equation*}
\begin{split}
I_3=&\int_0^T\sigma\int_{\mathbb{R}^2}
{\bf v}_t\nabla\tilde{u} dxdt =\int_0^T\sigma\int_{\mathbb{R}^2}
|\nabla\tilde{u}|^2 dxdt\le \int_0^T\int_{\mathbb{R}^2}
|\nabla\tilde{u}|^2 dxdt\le 2\theta_0,
\end{split}
\end{equation*}
where we have used the second equation of (\ref{pw7-hptr}). Substituting \eqref{pw7-es1}, \eqref{pw7-es2}, \eqref{pw7-es23} and the above estimates for $I_i (i=1,2,3)$ into \eqref{pw7-es-ut}, we have
\begin{equation}\label{pw7-es6}
\begin{split}
\sigma\norm{\nabla\tilde{u}}^2+4\int_0^T\sigma\norm{\tilde{u}_t}^2dt\le 4\lambda\int_0^T\sigma^2
\|\nabla\tilde{u}_t\|^2dt+C\theta_0.
\end{split}
\end{equation}

{\bf Step 2}. Next, we need to obtain the estimate of
$\displaystyle\int_0^T\sigma^2\norm{\nabla\tilde{u}_t}^2dt$. To this end, we differentiate $\eqref{pw7-hptr}$ with respect time $t$ to get
\begin{equation}\label{hptr2}
\begin{cases}
\tilde{u}_{tt}-\Delta\tilde{u}_{t}=\nabla\cdot(\tilde{u}{\bf v})_{t}+\nabla\cdot {\bf v}_{t},\\
{\bf v}_{tt}-\nabla\tilde{u}_t=0.
\end{cases}
\end{equation}
Multiplying the first equation of $\eqref{hptr2}$ by $\sigma^2\tilde{u}_t$ and the
second by $\sigma^2{\bf v}_t$, adding the results followed by an integration over $\R^2\times[0, t]$, we have
\begin{equation}\label{pw7-es-utt}
\begin{split}
&\frac{\sigma^2}{2}\norm{\tilde{u}_t}^2+\frac{\sigma^2}{2}\norm{{\bf v}_t}^2+\int_0^T\sigma^2\norm{\nabla\tilde{u}_t}^2dt\\
=&\int_0^{\sigma(T)}\sigma\norm{\tilde{u}_t}^2dt+\int_0^{\sigma(T)}\sigma\norm{{\bf v}_t}^2dt-\int_0^T\sigma^2\int_{\mathbb{R}^2}(\tilde{u}{\bf v})_t \nabla\tilde{u}_{t}dxdt\\
\le &\int_0^T\sigma\norm{\tilde{u}_t}^2dt+\int_0^T\sigma\norm{\nabla\tilde{u}}^2dt-\int_0^T\sigma^2\int_{\mathbb{R}^2}\tilde{u}{\bf v}_t \nabla\tilde{u}_tdxdt-\int_0^T\sigma^2\int_{\mathbb{R}^2}\tilde{u}_t {\bf v} \nabla\tilde{u}_tdxdt\\
\le &C\theta_0+\int_0^T\sigma\norm{\tilde{u}_t}^2dt-\int_0^T\sigma^2\int_{\mathbb{R}^2}\tilde{u}\nabla\tilde{u} \nabla\tilde{u}_tdxdt-\int_0^T\sigma^2\int_{\mathbb{R}^2}\tilde{u}_t {\bf v} \nabla\tilde{u}_tdxdt,
\end{split}
\end{equation}
where we have used the second equation of (\ref{pw7-hptr}) and \eqref{pw7-L^4}. For the third term on the right-hand of the above inequality, we estimate it by Cauchy-Schwarz inequality, \eqref{G-Ns} and  \eqref{pw7-L^4} as
\begin{equation}\label{pw7-es-l2}
\begin{split}
\int_0^T\sigma^2\int_{\mathbb{R}^2}\tilde{u}\nabla\tilde{u} \nabla\tilde{u}_tdxdt\le& \frac{1}{4}\int_0^T\sigma^2 \norm{\nabla\tilde{u}_t}^2 dt+ \int_0^T\sigma^2\norm{\tilde{u}\nabla\tilde{u}}^2dt\\
\le&\frac{1}{4}\int_0^T\sigma^2 \norm{\nabla\tilde{u}_t}^2 dt+ \int_0^T\sigma^2\norm{\tilde{u}}_{L^4}^2\norm{\nabla\tilde{u}}_{L^4}^2dt\\
\le&\frac{1}{4}\int_0^T\sigma^2 \norm{\nabla\tilde{u}_t}^2 dt+ C\int_0^T\sigma^2\norm{\tilde{u}}\norm{\nabla\tilde{u}}\norm{\nabla\tilde{u}}_{L^4}^2dt\\
\le&\frac{1}{4}\int_0^T\sigma^2 \norm{\nabla\tilde{u}_t}^2 dt+C\int_0^T\sigma^2\norm{\nabla\tilde{u}}^2dt+ C\int_0^T\sigma^2\norm{\tilde{u}}^2\norm{\nabla\tilde{u}}_{L^4}^4dt\\
\le&C\theta_0+\frac{1}{4}\int_0^T\sigma^2 \norm{\nabla\tilde{u}_t}^2 dt+ C\theta_0\int_0^T\sigma^2\norm{\nabla\tilde{u}}_{L^4}^4dt.
\end{split}
\end{equation}
Furthermore we have from \eqref{G-Ns} and \eqref{pw7-evf} that
\begin{equation}\label{pw7-u40-12}
\begin{split}
\norm{\nabla \tilde{u}}_{L^4}\le& \norm{{\bf F}- \tilde{u}{\bf v}-{\bf v}}_{L^4}\\
\le&\norm{{\bf F}}_{L^4}+\norm{\tilde{u}{\bf v}}_{L^4}+\norm{{\bf v}}_{L^4}\\
\le& C\big(\norm{{\bf F}}^\frac12\norm{\nabla {\bf F}}^\frac12+\norm{\tilde{u}}_{L^\infty}\norm{{\bf v}}_{L^{4}}+\norm{{\bf v}}_{L^4}\big)\\
\le& C\big(\norm{{\bf F}}^\frac12\norm{\nabla {\bf F}}^\frac12+\norm{\tilde{u}}_{L^4}^\frac12\norm{\nabla\tilde{u}}_{L^4}^\frac12\norm{{\bf v}}_{L^{4}}+\norm{{\bf v}}_{L^4}\big).
\end{split}
\end{equation}
Thus, by \eqref{pw7-assup} and \eqref{pw7-u40-12}, we have
\begin{equation}\label{pw7-u40-1}
\begin{split}
&\int_0^T\sigma^2\norm{\nabla\tilde{u}}_{L^4}^4dt\\
\le& C\int_0^T\sigma^2\left(\norm{{\bf F}}^2\norm{\nabla {\bf F}}^2+\norm{\tilde{u}}_{L^4}^2\norm{\nabla\tilde{u}}_{L^4}^2\norm{{\bf v}}_{L^{4}}^4+\norm{{\bf v}}_{L^4}^4\right)dt\\
\le& CA_3(T)+C\int_0^T\sigma^2\norm{{\bf F}}^2\norm{\nabla {\bf F}}^2dt+C\int_0^T\sigma^2\norm{\tilde{u}}_{L^4}^2\norm{\nabla\tilde{u}}_{L^4}^2\norm{{\bf v}}_{L^{4}}^4dt\\
=&C\theta_0^{\eta_0}+ I_4+I_5.
\end{split}
\end{equation}
To control $I_4$, we need to estimate $\norm{{\bf F}}$ and $\norm{\nabla {\bf F}}$ first. By \eqref{G-Ns}, \eqref{pw7-evf}, Cauchy-Schwarz inequality, we have from \eqref{pw7-ass1}, \eqref{pw7-assup} and \eqref{pw7-L^4} that
\begin{equation}\label{pw7-f22-0}
\begin{split}
\norm{{\bf F}}\le& \norm{\nabla\tilde{u}+\tilde{u}{\bf v}+{\bf v}}\\
\le& \norm{\nabla\tilde{u}}+\norm{\tilde{u}{\bf v}}+\norm{{\bf v}}\\
\le& \norm{\nabla\tilde{u}}+\norm{\tilde{u}}_{L^4}\norm{{\bf v}}_{L^4}+\norm{{\bf v}}\\
\le& \norm{\nabla\tilde{u}}+C\norm{\tilde{u}}^\frac12\norm{\nabla\tilde{u}}^\frac12\norm{{\bf v}}_{L^4}+\norm{{\bf v}}\\
\le& \norm{\nabla\tilde{u}}+C\norm{\tilde{u}}\norm{{\bf v}}_{L^4}^2+\norm{{\bf v}}\\
\le& \norm{\nabla\tilde{u}}+C\theta_0^\frac{\eta_0+1}{2}+C\theta_0^\frac12\\
\le& \norm{\nabla\tilde{u}}+C\theta_0^\frac12,
\end{split}
\end{equation}
where we have used the assumption $\theta_0<1$ and the fact $\eta_0>0$.

By \eqref{pw7-f2p} and H\"{o}lder inequality, we have
\begin{equation}\label{pw7-f2pf-0}
\begin{split}
\norm{\nabla {\bf F}}
\le C \norm{\tilde{u}_t}+ C\norm{\nabla^{\bot}\tilde{u}\cdot {\bf v}}
\le C \norm{\tilde{u}_t}+ C\norm{\nabla\tilde{u}}_{L^4}\norm{{\bf v}}_{L^4}.
\end{split}
\end{equation}
Then by \eqref{pw7-ass1}, \eqref{pw7-assup}, \eqref{pw7-f22-0}, \eqref{pw7-f2pf-0} and Cauchy-Schwarz inequality, we estimate $I_4$ as
\begin{equation}\label{pw7-f2pf-1}
\begin{split}
I_4
\le& C\int_0^T\sigma(\norm{\nabla\tilde{u}}^2+\theta_0)\sigma(\norm{\tilde{u}_t}^2+ \norm{\nabla\tilde{u}}^2_{L^4}\norm{{\bf v}}^2_{L^4}) dt\\
\le& C\left(A_2(T)+\theta_0\right)\int_0^T\sigma(\norm{\tilde{u}_t}^2+ \norm{\nabla\tilde{u}}^2_{L^4}\norm{{\bf v}}^2_{L^4}) dt\\
\le& C\theta_0^\frac12\int_0^T\sigma\norm{\tilde{u}_t}^2dt+C\theta_0^\frac12\int_0^T\sigma^2\norm{\nabla\tilde{u}}^4_{L^4}dt+C\theta_0^\frac12\int_0^T\norm{{\bf v}}^4_{L^4}dt\\
\le& C\theta_0^\frac12A_2(T)+C\theta_0^\frac12\int_0^T\sigma^2\norm{\nabla\tilde{u}}^4_{L^4}dt+C\theta_0^\frac12A_3(T)\\
\le& C\theta_0+C\theta_0^\frac12\int_0^T\sigma^2\norm{\nabla\tilde{u}}^4_{L^4}dt+C\theta_0^{\frac{1}{2}+\eta_0}\\
\le& C\theta_0^\frac12+C\theta_0^\frac12\int_0^T\sigma^2\norm{\nabla\tilde{u}}^4_{L^4}dt.
\end{split}
\end{equation}
Furthermore we use \eqref{G-Ns}, \eqref{pw7-assup} and \eqref{pw7-L^4} to estimate $I_5$ as
\begin{equation}\label{pw7-i-5}
\begin{split}
I_5=&C\int_0^T\sigma^2\norm{\tilde{u}}_{L^4}^2\norm{\nabla\tilde{u}}_{L^4}^2\norm{{\bf v}}_{L^{4}}^4dt\\
\le& \frac14\int_0^T\sigma^2\norm{\nabla\tilde{u}}_{L^4}^4dt+CA_3^2(T)\int_0^T\sigma^2\norm{\tilde{u}}_{L^4}^4dt\\
\le&\frac14\int_0^T\sigma^2\norm{\nabla\tilde{u}}_{L^4}^4dt+CA_3^2(T)\int_0^T\sigma^2\norm{\tilde{u}}^2\norm{\nabla\tilde{u}}^2dt\\
\le& \frac14\int_0^T\sigma^2\norm{\nabla\tilde{u}}_{L^4}^4dt+CA_1^2(T)A_3^2(T)\\
\le& \frac14\int_0^T\sigma^2\norm{\nabla\tilde{u}}_{L^4}^4dt+C\theta_0^{2(1+\eta_0)}\\
\le& \frac14\int_0^T\sigma^2\norm{\nabla\tilde{u}}_{L^4}^4dt+C\theta_0^{\eta_0}.
\end{split}
\end{equation}
Substituting \eqref{pw7-f2pf-1} and \eqref{pw7-i-5}  into \eqref{pw7-u40-1} and choosing $\theta_0$ small enough such that $C\theta_0^\frac12\le \frac14$, we get from \eqref{pw7-etas} that
\begin{equation}\label{pw7-u4l2}
\begin{split}
\int_0^T\sigma^2\norm{\nabla\tilde{u}}_{L^4}^4dt
\le C\theta_0^{\eta_0}+C\theta_0^{\frac12}\le C\theta_0^{\eta_0}.
\end{split}
\end{equation}
By \eqref{pw7-es-l2} and \eqref{pw7-u4l2}, we get
\begin{equation}\label{pw7-es-l2f}
\begin{split}
\bigg|\int_0^T\sigma^2\int_{\mathbb{R}^2}\tilde{u}\nabla\tilde{u} \nabla\tilde{u}_tdxdt\bigg|\le& C\theta_0^{1+\eta_0}+\frac{1}{4}\int_0^T\sigma^2 \norm{\nabla\tilde{u}_t}^2 dt\\
\le& C\theta_0+\frac{1}{4}\int_0^T\sigma^2 \norm{\nabla\tilde{u}_t}^2 dt.
\end{split}
\end{equation}
So far we have finished estimating the third term on the right hand side of \eqref{pw7-es-utt}. Next we proceed to estimate
the last term on the right-hand of \eqref{pw7-es-utt}. To this end, we use Cauchy-Schwarz inequality, \eqref{G-Ns} and \eqref{pw7-assup} to get
\begin{equation}\label{pw7-es-l3}
\begin{split}
\int_0^T\sigma^2\int_{\mathbb{R}^2}\tilde{u}_t {\bf v} \nabla\tilde{u}_tdxdt\le& \frac{1}{4}\int_0^T\sigma^2 \norm{\nabla\tilde{u}_t}^2 dt+ \int_0^T\sigma^2\norm{\tilde{u}_t {\bf v}}^2dt\\[2mm]
\le& \frac{1}{4}\int_0^T\sigma^2 \norm{\nabla\tilde{u}_t}^2 dt+ \int_0^T\sigma^2\norm{\tilde{u}_t}_{L^4}^2\norm{{\bf v}}_{L^4}^2dt\\[2mm]
\le& \frac{1}{4}\int_0^T\sigma^2 \norm{\nabla\tilde{u}_t}^2 dt+ C\int_0^T\sigma^2\norm{\tilde{u}_t}\norm{\nabla\tilde{u}_t}\norm{{\bf v}}_{L^4}^2dt\\[2mm]
\le& \frac{1}{2}\int_0^T\sigma^2 \norm{\nabla\tilde{u}_t}^2 dt+ C\int_0^T\sigma^2\norm{\tilde{u}_t}^2\norm{{\bf v}}_{L^4}^4dt\\[2mm]
\le& \frac{1}{2}\int_0^T\sigma^2 \norm{\nabla\tilde{u}_t}^2 dt+ C\theta_0^{\eta_0}\int_0^T\sigma^2\norm{\tilde{u}_t}^2dt.
\end{split}
\end{equation}
Substituting \eqref{pw7-es-l2f} and \eqref{pw7-es-l3} into \eqref{pw7-es-utt}, we get from $0\le\sigma\le 1$ and $\theta_0\le 1$ that
\begin{equation*}
\begin{split}
\sigma^2\norm{\tilde{u}_t}^2+\sigma^2\norm{{\bf v}_t}^2+\int_0^T\sigma^2\norm{\nabla\tilde{u}_t}^2dt
\le C\int_0^T\sigma\norm{\tilde{u}_t}^2dt+ C\theta_0,
\end{split}
\end{equation*}
which, along with \eqref{pw7-es6}, immediately leads to
\begin{equation*}
\begin{split}
&\sigma^2\norm{\tilde{u}_t}^2+\sigma^2\norm{{\bf v}_t}^2+\int_0^T\sigma^2\norm{\nabla\tilde{u}_t}^2dt
\le C\lambda\int_0^T\sigma^2
\|\nabla\tilde{u}_t\|^2dt+ C\theta_0.
\end{split}
\end{equation*}
Choosing $\lambda$ small such that $C\lambda\le \frac12$, we have
\begin{equation*}
\begin{split}
&\sigma^2\norm{\tilde{u}_t}^2+\sigma^2\norm{{\bf v}_t}^2+\int_0^T\sigma^2\norm{\nabla\tilde{u}_t}^2dt
\le C\theta_0,
\end{split}
\end{equation*}
which, together with \eqref{pw7-es6}, gives
\begin{equation}\label{pw7-es9}
\begin{split}
\sigma\norm{\nabla\tilde{u}}^2+\sigma^2\norm{\tilde{u}_t}^2+\sigma^2\norm{{\bf v}_t}^2+\int_0^T\sigma\norm{\tilde{u}_t}^2dt+\int_0^T\sigma^2\norm{\nabla\tilde{u}_t}^2dt
\le  C\theta_0.
\end{split}
\end{equation}
Taking $\theta_0$ small enough such that $C\theta_0^\frac12\le 1$, we have from \eqref{pw7-es9} that
\begin{equation*}
\begin{split}
\sigma\norm{\nabla\tilde{u}}^2+\sigma^2\norm{\tilde{u}_t}^2+\sigma^2\norm{{\bf v}_t}^2+\int_0^T\sigma\norm{\tilde{u}_t}^2dt+\int_0^T\sigma^2\norm{\nabla\tilde{u}_t}^2dt
\le  \theta_0^\frac12.
\end{split}
\end{equation*}
 Thus, the proof of Lemma \ref{pw7-leUt} is completed.
\end{proof}
Next we shall derive the {\it a priori} estimate of $\norm{{\bf{v}}}_{L^4}$. For this, we first derive some estimates that will be used later.
\begin{lemma}\label{3.4} Let the conditions of Theorem \ref{pw7-th} hold and $(\tilde{u}, {\bf v})$ be a smooth solution of \eqref{pw7-hptr}-\eqref{pw8mde0} satisfying \eqref{pw7-assup}. Then it holds that
\begin{equation*}
\begin{split}
\sigma^4\norm{\tilde{u}}^4_{L^\infty}\le C\theta_0.
\end{split}
\end{equation*}
Furthermore, for any $t\in(\sigma(T),T]$, it holds that
\begin{equation}\label{pw7-ulin}
\begin{split}
-\frac 14\leq\tilde{u}(x,t)\le \frac 14.
\end{split}
\end{equation}
\end{lemma}
\begin{proof}
\par
By Gagliardo-Nirenberg inequality \eqref{G-Ns}, we have from \eqref{pw7-L^4} and \eqref{pw7-L^2} that
\begin{equation}\label{pw7-u4}
\begin{split}
\sigma^4\norm{\tilde{u}}^4_{L^\infty}\le& C\sigma^4\norm{\tilde{u}}_{L^4}^2\norm{\nabla\tilde{u}}_{L^4}^2\\[2mm]
\le& C\sigma^4\norm{\tilde{u}}\norm{\nabla\tilde{u}}\norm{\nabla\tilde{u}}_{L^4}^2\\[2mm]
\le& C\sigma^{\frac32}\norm{\tilde{u}}\left(\sigma^{\frac12}\norm{\nabla\tilde{u}}\right)\sigma^{2}\norm{\nabla\tilde{u}}_{L^4}^2\\[2mm]
\le& C\sigma^{\frac32}A_1^\frac12(T)A_2^\frac12(T)\sigma^{2}\norm{\nabla\tilde{u}}_{L^4}^2\\[2mm]
\le& C\theta_0^\frac34\sigma^{2}\norm{\nabla\tilde{u}}_{L^4}^2.
\end{split}
\end{equation}
With \eqref{G-Ns}, \eqref{pw7-etas}, \eqref{pw7-L^4} and \eqref{pw7-u40-12}, we have the following estimates
\begin{equation}\label{pw7-up-1}
\begin{split}
\norm{\nabla \tilde{u}}_{L^4}
\le&  C\big(\norm{{\bf F}}^\frac12\norm{\nabla {\bf F}}^\frac12+\norm{\tilde{u}}_{L^4}^\frac12\norm{\nabla\tilde{u}}_{L^4}^\frac12\norm{{\bf v}}_{L^{4}}+\norm{{\bf v}}_{L^4}\big)\\
\le& \frac12\norm{\nabla {\bf F}}+C\norm{{\bf F}}+\frac12\norm{\nabla \tilde{u}}_{L^4}+C\norm{\tilde{u}}_{L^4}\norm{{\bf v}}^2_{L^4}+C\norm{{\bf v}}_{L^4}\\
\le&\frac12\norm{\nabla {\bf F}}+C\norm{{\bf F}}+\frac12\norm{\nabla \tilde{u}}_{L^4}+C\norm{\tilde{u}}^\frac12\norm{\nabla\tilde{u}}^\frac12\norm{{\bf v}}_{L^{4}}^2+C\norm{{\bf v}}_{L^4}\\
\le&\frac12\norm{\nabla {\bf F}}+C\norm{{\bf F}}+\frac12\norm{\nabla \tilde{u}}_{L^4}+C\norm{\nabla\tilde{u}}+C\norm{\tilde{u}}\norm{{\bf v}}_{L^{4}}^4+C\norm{{\bf v}}_{L^4}\\
\le&\frac12\norm{\nabla {\bf F}}+C\norm{{\bf F}}+\frac12\norm{\nabla \tilde{u}}_{L^4}+C\norm{\nabla\tilde{u}}+CA_1^\frac12(T)A_3(T)+CA_3^\frac14(T)\\
\le&\frac12\norm{\nabla {\bf F}}+C\norm{{\bf F}}+\frac12\norm{\nabla \tilde{u}}_{L^4}+C\norm{\nabla\tilde{u}}+C\theta_0^{\frac12+\eta_0}+C\theta_0^{\frac{\eta_0}{4}}\\
\le&\frac12\norm{\nabla {\bf F}}+C\norm{{\bf F}}+\frac12\norm{\nabla \tilde{u}}_{L^4}+C\norm{\nabla\tilde{u}}+C\theta_0^{\frac{\eta_0}{4}}.
\end{split}
\end{equation}
Then by \eqref{pw7-assup}, \eqref{pw7-f22-0} and \eqref{pw7-f2pf-0}, one has
\begin{equation}\label{pw7-f22}
\begin{split}
\norm{{\bf F}}\le C\norm{\nabla\tilde{u}}+C\theta_0^\frac12
\end{split}
\end{equation}
and
\begin{equation}\label{pw7-f2pf}
\begin{split}
\norm{\nabla {\bf F}}\le
C \norm{\tilde{u}_t}+C\theta_0^{\frac{\eta_0}{4}}\norm{\nabla\tilde{u}}_{L^4}.
\end{split}
\end{equation}
Substituting \eqref{pw7-f22} and \eqref{pw7-f2pf} into \eqref{pw7-up-1}, we get
\begin{equation}\label{pw7-up-2}
\begin{split}
\norm{\nabla \tilde{u}}_{L^4}\le&C\theta_0^\frac12+C\norm{\nabla\tilde{u}}+C \norm{\tilde{u}_t}+ C\theta_0^{\frac{\eta_0}{4}}\norm{\nabla\tilde{u}}_{L^4}+C\theta_0^{\frac{\eta_0}{4}}.
\end{split}
\end{equation}
Choosing $\theta_0$ small enough such that $C\theta_0^{\frac{\eta_0}{4}}\le \frac12$, using \eqref{pw7-etas} and \eqref{pw7-up-2}, we end up with
\begin{equation}\label{pw7-up-3}
\begin{split}
\norm{\nabla \tilde{u}}_{L^4}
\le C\theta_0^{\frac{\eta_0}{4}}+C\norm{\nabla\tilde{u}}+C \norm{\tilde{u}_t},
\end{split}
\end{equation}
which, along with \eqref{pw7-f2pf} and $\theta_0\le 1$,  yields
\begin{equation}\label{pw7-f2-pf}
\begin{split}
\norm{\nabla {\bf F}}\le C\theta_0^{\frac{\eta_0}{2}}+C \norm{\tilde{u}_t}+ C\norm{\nabla\tilde{u}}.
\end{split}
\end{equation}
Then by the fact $0\le \sigma\le 1$, \eqref{pw7-etas}, \eqref{pw7-L^2} and \eqref{pw7-up-3}, it holds that
\begin{equation}\label{pw7-u42}
\begin{split}
\sigma^{2}\norm{\nabla\tilde{u}}_{L^4}^2\le& C\sigma^{2}\theta_0^{\frac{\eta_0}{2}}+C\sigma^{2}\norm{\nabla\tilde{u}}^2+C\sigma^{2}\norm{\tilde{u}_t}^2\\
\le& C\theta_0^{\frac{\eta_0}{2}}+CA_2(T)\\
\le& C\theta_0^{\frac{\eta_0}{2}}+C\theta_0^{\frac12}\\
\le& C\theta_0^{\frac{\eta_0}{2}}.
\end{split}
\end{equation}
With \eqref{pw7-u42}, \eqref{pw7-u4} is updated as
\begin{equation*}
\begin{split}
\sigma^4\norm{\tilde{u}}^4_{L^\infty}\le C\theta_0^\frac{3+2\eta_0}{4}\le C\theta_0^\frac34.
\end{split}
\end{equation*}
Choosing $\theta_0$ small enough such that $C\theta_0^\frac34\le 4^{-4}$ and using the fact $\sigma(t)=1$ for $t\in(\sigma(T),T]$, we have
\begin{equation}\label{3.38}
\begin{split}
\sup\limits_{t\in[\sigma(T), T]}\abs{\tilde{u}(x,t)}\le \frac{1}{4},
\end{split}
\end{equation}
which implies
\begin{equation*}
\begin{split}
-\frac 14\leq\tilde{u}(x,t)\le \frac 14,\ {\rm for } \ t\in(\sigma(T),T].
\end{split}
\end{equation*}
 Thus, the proof of Lemma \ref{3.4} is completed.
\end{proof}

\begin{lemma}\label{pw7-v6-6}
Let the conditions of Theorem \ref{pw7-th} hold and $(\tilde{u}, {\bf v})$ be a smooth solution of \eqref{pw7-hptr}-\eqref{pw8mde0} satisfying \eqref{pw7-assup}. Then it holds that
\begin{equation*}
\begin{split}
 \sup\limits_{t\in[0,T]}\int_{\R^2}{\bf v}^{4}dx+\int_{0}^t\int_{\R^2}{\bf v}^{4}dxdt\le \theta_0^{\eta_0},
\end{split}\end{equation*}
where $\D\eta_0\triangleq \frac{p_0-4}{2(p_0-2)}$.
\end{lemma}
\begin{proof}
\par First it follows from ${\bf v}_t=\nabla\tilde{u}$ and \eqref{pw7-evf} that
\begin{equation}\label{pw7-vc}
\begin{split}
{\bf v}_t+(\tilde{u}+1){\bf v}={\bf F}.
\end{split}\end{equation}
Multiplying \eqref{pw7-vc} by $|{\bf v}|^{2}\bf v$ and integrating the resulting equality over $\R^2$, one has
\begin{equation}\label{pw7-vc1}
\begin{split}
\frac1{4} \left(\int_{\R^2}{\bf v}^{4}dx\right)_t+\int_{\R^2}(\tilde{u}+1){\bf v}^{4}dx=\int_{\R^2}{\bf F}|{\bf v}|^{2}{\bf v}dx.
\end{split}\end{equation}
Integrating the above equality over $[\sigma(T), T]$ and using \eqref{pw7-ulin}, we have
\begin{equation}\label{pw7-vc24}
\begin{split}
\frac1{4}\int_{\R^2}{\bf v}^{4}dx+\frac34\int_{\sigma(T)}^T\int_{\R^2}{\bf v}^{4}dxdt\le \sup\limits_{t\in[0,\sigma(T)]}\left(\frac1{4}\int_{\R^2}{\bf v}^{4}dx\right)+ \int_{\sigma(T)}^T\int_{\R^2}|{\bf F}||{\bf v}|^{3}dxdt.
\end{split}\end{equation}
By the interpolation inequality and \eqref{pw7-assup}, for any $2<p<p_0$, we have
\begin{equation*}
\begin{split}
\|{\bf v}\|_{L^p}\le \|{\bf v}\|^{\frac{2(p_0-p)}{p(p_0-2)}}\|{\bf v}\|_{L^{p_0}}^{\frac{p_0(p-2)}{p(p_0-2)}}\le C \theta_0^{\frac{(p_0-p)}{p(p_0-2)}}M^{\frac{p_0(p-2)}{p(p_0-2)}},
\end{split}
\end{equation*}
which implies that
\begin{equation}\label{pw7-v0-lp}
\begin{split}
\|{\bf v}\|_{L^4}\le C \theta_0^{\frac{(p_0-4)}{2(p_0-2)}}M^{\frac{p_0}{2(p_0-2)}}\le C(M)\theta_0^{\frac{(p_0-4)}{2(p_0-2)}}\le C(M)\theta_0^{\eta_0}.
\end{split}
\end{equation}
We need to further estimate the last term in \eqref{pw7-vc24}. By the Young
inequality, we have that
\begin{equation*}
\begin{split}
\int_{\sigma(T)}^T\int_{\R^2}|{\bf F}||{\bf v}|^{3}dxdt\le \frac 14\int_{\sigma(T)}^T\int_{\R^2}{\bf v}^{4}dxdt+C\int_{\sigma(T)}^T\int_{\R^2}|{\bf F}|^{4}dxdt,
\end{split}\end{equation*}
which updates \eqref{pw7-vc24} as
\begin{equation}\label{pw7-vc3}
\begin{split}
\frac1{4}\int_{\R^2}{\bf v}^{4}dx+\frac12\int_{\sigma(T)}^T\int_{\R^2}{\bf v}^{4}dxdt\le& \sup\limits_{t\in[0,\sigma(T)]}\left(\frac1{4}\int_{\R^2}{\bf v}^{4}dx\right)+ C\int_{\sigma(T)}^T\int_{\R^2}|{\bf F}|^{4}dxdt\\
\le&C(M)\theta_0^{4\eta_0} +C\int_{\sigma(T)}^T\int_{\R^2}|{\bf F}|^{4}dxdt,
\end{split}\end{equation}
where we have used \eqref{pw7-v0-lp}. By Gagliardo-Nirenberg inequality \eqref{G-Ns}, one has
\begin{equation}\label{pw7-F6}
\begin{split}
\int_{\sigma(T)}^T\norm{{\bf F}}_{L^4}^4dt\le& C\int_{\sigma(T)}^T\norm{{\bf F}}^2\norm{\nabla {\bf F}}^2dt\\
\le& C\left(\sup_{\sigma(T)\le t\le T}\norm{{\bf F}}^2\right)\int_{\sigma(T)}^T\norm{\nabla {\bf F}}^2dt.
\end{split}\end{equation}
It follows from \eqref{pw7-L^2} and \eqref{pw7-f22-0} that
\begin{equation}\label{pw7-Fl2f}
\begin{split}
\sup_{\sigma(T)\le t\le T}\norm{{\bf F}}^2
\le& C\sup_{\sigma(T)\le t\le T}\sigma\norm{\nabla\tilde{u}}^2+C\theta_0\le C\theta_0^\frac12+C\theta_0\le C\theta_0^\frac12,
\end{split}
\end{equation}
where the fact that $\sigma(t)=1$ for  $t\geq \sigma(T)$ has been used. By \eqref{pw7-L^2} and \eqref{pw7-f2pf-0}, we have
\begin{equation}\label{pw7-deFl}
\begin{split}
\int_{\sigma(T)}^T\norm{\nabla {\bf F}}^2dt
\le&C \int_{\sigma(T)}^T\sigma\norm{\tilde{u}_t}^2dt+ C\int_{\sigma(T)}^T\sigma\norm{\nabla\tilde{u}}^2_{L^4}\norm{{\bf v}}^2_{L^4}dt\\
\le&C\theta_0^\frac12+ C\int_{\sigma(T)}^T\sigma\norm{\nabla\tilde{u}}^2_{L^4}\norm{{\bf v}}^2_{L^4}dt.
\end{split}
\end{equation}
Then with \eqref{pw7-u4l2}, Cauchy-Schwarz inequality and \eqref{pw7-assup}, we get
\begin{equation*}
\begin{split}
& C\int_{\sigma(T)}^T\sigma\norm{\nabla\tilde{u}}^2_{L^4}\norm{{\bf v}}^2_{L^4}dt
\le C\int_{\sigma(T)}^T\sigma^2\norm{\nabla\tilde{u}}^4_{L^4}dt+ C\int_{\sigma(T)}^T\norm{{\bf v}}^4_{L^4}dt
\le C\theta_0^{\eta_0},
\end{split}
\end{equation*}
which, along with \eqref{pw7-deFl} gives
\begin{equation*}
\begin{split}
\int_{\sigma(T)}^T\norm{\nabla {\bf F}}^2dt
\le&C\theta_0^{\eta_0}.
\end{split}
\end{equation*}
This together with \eqref{pw7-F6} and \eqref{pw7-Fl2f} leads to
\begin{equation}\label{pw7-F4-0}
\begin{split}
\int_{\sigma(T)}^T\norm{{\bf F}}_{L^4}^4dt\le C\theta_0^{\frac12+\eta_0}.
\end{split}\end{equation}
Substituting \eqref{pw7-F4-0} into \eqref{pw7-vc3} to get
\begin{equation*}
\begin{split}
\sup\limits_{t\in[\sigma(T),T]}\int_{\R^2}{\bf v}^{4}dx+\int_{\sigma(T)}^T\int_{\R^2}{\bf v}^{4}dxdt
\le C(M)\theta_0^{4\eta_0} +C\theta_0^{\frac12+\eta_0}.
\end{split}\end{equation*}
For $0\leq t\leq\sigma(T)$, we have from \eqref{pw7-v0-lp} that
\begin{equation*}
\begin{split}
 \sup\limits_{t\in[0,\sigma(T)]}\int_{\R^2}{\bf v}^{4}dx+\int_0^{\sigma(T)}\int_{\R^2}{\bf v}^{4}dxdt\le 2\sup\limits_{t\in[0,\sigma(T)]}\int_{\R^2}{\bf v}^{4}dx\le   C(M)\theta_0^{4\eta_0}.
\end{split}\end{equation*}
Coupling the above two inequalities together, we arrive at
\begin{equation}\label{pw7-vcf}
\begin{split}
 \sup\limits_{t\in[0,T]}\int_{\R^2}{\bf v}^{4}dx+\int_{0}^t\int_{\R^2}{\bf v}^{4}dxdt\le C(M)\theta_0^{4\eta_0} +C\theta_0^{\frac12+\eta_0}\le \theta_0^{\eta_0},
\end{split}\end{equation}
provided that $C(M)\theta_0^{3\eta_0}\le \frac12$ and $C\theta_0^\frac12\le \frac12$. This completes the proof of Lemma \ref{pw7-v6-6}.
\end{proof}

\begin{lemma}\label{pw7-v66} Let the conditions of Theorem \ref{pw7-th} hold and $(\tilde{u}, {\bf v})$ be a smooth solution of \eqref{pw7-hptr}-\eqref{pw8mde0} satisfying \eqref{pw7-assup}, then it holds that
\begin{equation}\label{pw7-vp0f}
\begin{split}
\sup\limits_{t\in[0,T]}\|{\bf v}\|_{L^{p_0}}\le 3M.
\end{split}\end{equation}
\end{lemma}
\begin{proof}
Multiplying \eqref{pw7-vc} by $|{\bf v}|^{{p_0}-2}{\bf v}(4<{p_0}<\infty)$ and integrating the resulting equality over $\R^2$, one has
\begin{equation}\label{pw7-vc1-1}
\begin{split}
\frac1{{p_0}} \left(\int_{\R^2}|{\bf v}|^{{p_0}}dx\right)_t+\int_{\R^2}(\tilde{u}+1)|{\bf v}|^{{p_0}}dx=\int_{\R^2}{\bf F}|{\bf v}|^{{p_0}-2}{\bf v}dx.
\end{split}\end{equation}
For $t\in[0,\sigma(T)]$, it is known that $u=\tilde{u}+1\geq0$. Then the H\"older inequality yields that
$
\frac{1}{p_0}\left(\|{\bf v}\|_{L^{p_0}}^{p_0}\right)_t\le\|{\bf F}\|_{L^{p_0}}\norm{{\bf v}}_{L^{p_0}}^{p_0-1},
$
which hence
leads to
\begin{equation*}
\begin{split}
\left(\|{\bf v}\|_{L^{p_0}}\right)_t\le \|{\bf F}\|_{L^{p_0}}.
\end{split}\end{equation*}
Integrating the above equality over $[0,\sigma(T)]$, we have
\begin{equation}\label{pw7-vc10}
\begin{split}
\sup\limits_{t\in[0,\sigma(T)]}\|{\bf v}\|_{L^{p_0}}\le \|{\bf v}_0\|_{L^{p_0}}+\int_0^{\sigma(T)}\|{\bf F}\|_{L^{p_0}}dt.
\end{split}\end{equation}
The last term in \eqref{pw7-vc10} can be estimated by the Gagliardo-Nirenberg inequality \eqref{G-Ns} as
\begin{equation}\label{pw7-f1}
\begin{split}
\int_0^{\sigma(T)}\norm{{\bf F}}_{L^{{p_0}}}dt\le& C\int_0^{\sigma(T)}\norm{{\bf F}}^\frac{2}{p_0}\norm{\nabla {\bf F}}^\frac{{p_0}-2}{{p_0}}dt.
\end{split}\end{equation}
By \eqref{pw7-L^2} and \eqref{pw7-f22}, one gets
\begin{equation}\label{pw7-f2f}
\begin{split}
\sup\limits_{t\in[0,\sigma(T)]}\sigma^{\frac12}\norm{{\bf F}}\le& C\sigma^{\frac12}\norm{\nabla\tilde{u}}+C\sigma^{\frac12}\theta_0^\frac12\le C\theta_0^\frac14.
\end{split}
\end{equation}
It then follows from \eqref{pw7-etas}, \eqref{pw7-L^4}, \eqref{pw7-L^2}, \eqref{pw7-f2pf-0} and \eqref{pw7-u4l2} that
\begin{equation}\label{pw7-f233}
\begin{split}
\int_0^{T}\sigma\norm{\nabla {\bf F}}^2dt\le&  C \int_0^{T}\sigma\norm{\tilde{u}_t}^2dt+ C\int_0^{T}\sigma\norm{\nabla\tilde{u}}^2_{L^4}\norm{{\bf v}}^2_{L^4}dt\\
\le& C \int_0^{T}\sigma\norm{\tilde{u}_t}^2dt+ C\int_0^{T}\left(\sigma^2\norm{\nabla\tilde{u}}^4_{L^4}+\norm{{\bf v}}^4_{L^4}\right)dt\\
\le& C\theta_0^\frac12+C\theta_0^{\eta_0}\le C\theta_0^{\eta_0}.
\end{split}
\end{equation}
Substituting \eqref{pw7-f2f}-\eqref{pw7-f233} into \eqref{pw7-f1} and using \eqref{pw7-L^2} with H\"older inequality, we end up with
\begin{equation*}
\begin{split}
\int_0^{\sigma(T)}\norm{{\bf F}}_{L^{{p_0}}}dt\le& C\int_0^{\sigma(T)}\norm{{\bf F}}^\frac{2}{p_0}\norm{\nabla {\bf F}}^\frac{{p_0}-2}{{p_0}}dt\\
\le& \int_0^{\sigma(T)}\left(\sigma^\frac12\norm{{\bf F}}\right)^\frac{2}{{p_0}}\left(\sigma\norm{\nabla {\bf F}}^2\right)^\frac{{p_0}-2}{2{p_0}}\sigma^{-\frac12}dt\\
\le& C\theta_0^\frac{1}{2{p_0}}\int_0^{\sigma(T)}\left(\sigma\norm{\nabla {\bf F}}^2\right)^\frac{{p_0}-2}{2{p_0}}\sigma^{-\frac12}dt\\
\le& C\theta_0^\frac{1}{2{p_0}}\left(\int_0^{\sigma(T)}\sigma\norm{\nabla {\bf F}}^2dt\right)^\frac{{p_0}-2}{2{p_0}}\left(\int_0^{\sigma(T)}\sigma^{-\frac12\times\frac{2{p_0}}{p_0+2}}dt\right)^\frac{{p_0}+2}{2{p_0}}\\
\le& C\theta_0^\frac{1+\eta_0({p_0}-2)}{2{p_0}}\left(\int_0^{\sigma(T)}t^{-\frac{{p_0}}{p_0+2}}dt\right)^\frac{{p_0}+2}{2{p_0}}\\
\le& C\theta_0^\frac{p_0-2}{4p_0},
\end{split}\end{equation*}
which, along with \eqref{pw7-vc10}, gives
\begin{equation}\label{pw7-vc10f}
\begin{split}
\sup\limits_{t\in[0,\sigma(T)]}\|{\bf v}\|_{L^{p_0}}\le M+ C\theta_0^\frac{p_0-2}{4p_0}\le \frac{3M}{2},
\end{split}\end{equation}
provided that $ C\theta_0^\frac{p_0-2}{4p_0}\le \frac{M}{2}$.\\

Next consider the estimate of $\|{\bf v}\|_{L^{p_0}}$ for  $t\in[\sigma(T), T]$. Indeed integrating \eqref{pw7-vc1-1} over $[\sigma(T), T])$ and using $\tilde{u}\geq-\frac14$, we have
\begin{equation}\label{pw7-vc-2}
\begin{split}
\frac1{{p_0}}\int_{\R^2}|{\bf v}|^{{p_0}}dx+\frac34\int_{\sigma(T)}^T\int_{\R^2}|{\bf v}|^{{p_0}}dxdt\le \sup\limits_{t\in[0,\sigma(T)]}\left(\frac1{{p_0}}\int_{\R^2}|{\bf v}|^{{p_0}}dx\right)+ \int_{\sigma(T)}^T\int_{\R^2}|{\bf F}||{\bf v}|^{{p_0}-1}dxdt.
\end{split}
\end{equation}
We proceed to estimate the last term in \eqref{pw7-vc-2}. By the Young
inequality:
\begin{equation*}
\begin{split}
ab\le \epsilon a^q+(\epsilon q)^{-\frac rq}r^{-1}b^r, \ \ a, b\geq0,\ \epsilon>0, \ \ q, r>0, \ \ \frac 1q+\frac 1r=1,
\end{split}\end{equation*}
we see that
\begin{equation*}
\begin{split}
\int_{\sigma(T)}^T\int_{\R^2}|{\bf F}||{\bf v}|^{{p_0}-1}dxdt\le \frac 12\int_{\sigma(T)}^T\int_{\R^2}|{\bf v}|^{{p_0}}dxdt+\frac{1}{{p_0}}\left(\frac{2{p_0}-2}{{p_0}}\right)^{{p_0}-1}\int_{\sigma(T)}^T\int_{\R^2}|{\bf F}|^{{p_0}}dxdt,
\end{split}\end{equation*}
which updates \eqref{pw7-vc-2} as
\begin{equation}\label{pw7-vc2-0}
\begin{split}
&\int_{\R^2}|{\bf v}|^{{p_0}}dx+\frac {p_0}{4}\int_{\sigma(T)}^T\int_{\R^2}|{\bf v}|^{{p_0}}dxdt\\
\le& \sup\limits_{t\in[0,\sigma(T)]}\left(\int_{\R^2}|{\bf v}|^{{p_0}}dx\right)+ \left(\frac{2{p_0}-2}{{p_0}}\right)^{{p_0}-1}\int_{\sigma(T)}^T\int_{\R^2}|{\bf F}|^{p_0}dxdt.
\end{split}
\end{equation}
On the other hand,  from \eqref{G-Ns}, \eqref{pw7-fut} and \eqref{pw7-F6}, one has
\begin{equation}\label{pw7-f2}
\begin{split}
\int_{\sigma(T)}^T\norm{{\bf F}}_{L^{p_0}}^{p_0}dt\le& C\int_{\sigma(T)}^T\norm{{\bf F}}^{2}\norm{\nabla {\bf F}}^{p_0-2}dt
\le C\sup\limits_{t\in[\sigma(T),T]}\norm{{\bf F}}^{2}\norm{\nabla {\bf F}}^{p_0-4}\int_{\sigma(T)}^T\norm{\nabla {\bf F}}^2dt.
\end{split}\end{equation}
Then we have from \eqref{pw7-f2f} and \eqref{pw7-f233} that
\begin{equation*}
\begin{split}
\sup\limits_{t\in[\sigma(T),T]}\norm{{\bf F}}^2\le\sup\limits_{t\in[\sigma(T),T]}(\sigma\|{\bf F}\|^2)\le C\theta_0^\frac12
\end{split}\end{equation*}
and
\begin{equation*}
\begin{split}
\int_{\sigma(T)}^T\norm{\nabla {\bf F}}^2dt=\int_{\sigma(T)}^T\sigma\norm{\nabla {\bf F}}^2dt\le C\theta_0^{\eta_0},
\end{split}\end{equation*}
where we the fact $\sigma(t)=1$ for $t\in[\sigma(T),T]$ has been used.
By \eqref{pw7-etas}, \eqref{pw7-L^2} and \eqref{pw7-f2-pf}, we get
\begin{equation*}
\begin{split}
\sup\limits_{t\in[\sigma(T),T]}\norm{\nabla {\bf F}}\le C\theta_0^{\frac{\eta_0}{2}}+C\sigma\norm{\tilde{u}_t}+ C\sigma^\frac12\norm{\nabla\tilde{u}}
\le C\theta_0^{\frac{\eta_0}{2}}+\theta_0^\frac14\le C\theta_0^{\frac{\eta_0}{2}},
\end{split}\end{equation*}
which implies
\begin{equation*}
\begin{split}
\sup\limits_{t\in[\sigma(T),T]}\norm{\nabla {\bf F}}^{p_0-4}\le C\theta_0^{\frac{\eta_0(p_0-4)}{2}}.
\end{split}\end{equation*}
Substituting the above inequalities into \eqref{pw7-f2}, we have from \eqref{pw7-etas} that
\begin{equation}\label{pw7-f2f-1}
\begin{split}
\int_{\sigma(T)}^T\norm{{\bf F}}_{L^{p_0}}^{p_0}dt\le& C\theta_0^{\frac{\eta_0(p_0-2)+1}{2}}=C\theta_0^\frac{p_0-2}{4}.
\end{split}
\end{equation}
It follows from \eqref{pw7-vc10f}, \eqref{pw7-vc2-0} and \eqref{pw7-f2f-1} that
\begin{equation}\label{pw7-vc2-0f}
\begin{split}
\sup\limits_{t\in[\sigma(T),T]}\int_{\R^2}{\bf v}^{{p_0}}dx+\frac{p_0}{4}\int_{\sigma(T)}^T\int_{\R^2}{\bf v}^{{p_0}}dxdt
\le& \left(\frac{3M}{2}\right)^{p_0}+ C\left(\frac{2{p_0}-2}{{p_0}}\right)^{{p_0}-1}\theta_0^\frac{p_0-2}{4}\\
\le& \left(\frac{3M}{2}\right)^{p_0}+ C\theta_0^\frac{p_0-2}{4}2^{{p_0}-1}\\
\le& \left(3M\right)^{p_0},
\end{split}\end{equation}
provided that $\D  \theta_0^\frac{p_0-2}{4p_0}\le \frac{3M}{2}$.  Then, raising the power $\frac{1}{{p_0}}$ to both sides of \eqref{pw7-vc2-0f}, we obtain that
\begin{equation*}
\begin{split}
\sup\limits_{t\in[\sigma(T),T])}\|{\bf v}\|_{L^{p_0}}\le 3M,
\end{split}\end{equation*}
which combined to \eqref{pw7-vc10f} yields that
\begin{equation*}
\begin{split}
\sup\limits_{t\in[0,T]}\|{\bf v}\|_{L^{p_0}}\le 3M.
\end{split}\end{equation*}
Thus, the proof of Lemma \ref{pw7-v66} is finished.
\end{proof}
With the help of Lemmas \ref{pw7-th1}-\ref{pw7-v66}, we have
\begin{equation}\label{pw7-assupf}
\begin{split}
A_1(T)\leq \frac32\theta_0, \ \  A_2(T)
\le  \theta_0^\frac12,\ \ A_3(T)\le \theta_0^{\eta_0}, \ \
\sup_{t\in[0,T]}\|{\bf v}\|_{L^{p_0}}\le 3M,
\end{split}
\end{equation}
which closes the {\it a priori}  assumption \eqref{pw7-assup}.

\section{Proof of Theorem $\ref{pw7-th}$ }

In this section, we prove Theorem \ref{pw7-th} by constructing weak solutions as limits of approximate
smooth solutions. We start with the global-in-time existence of smooth solutions to \eqref{pw7-hptr}.
\begin{proposition}\label{pw7-global}
Assume the initial data satisfies $(\tilde{u}^\delta_0, {\bf v}^\delta_0)\in H^3(\R^2)$ and $\theta_0\le \varepsilon$. Then there exists a unique solution to the system
\eqref{pw7-hptr} such that $(\tilde{u}^\delta, {\bf v}^\delta)\in L^\infty([0,\infty), H^3)$.
\end{proposition}
\begin{proof}
\par
By Lemma \ref{pw7-local}, there exists a $T_*>0$ such that the
Cauchy problem \eqref{pw7-hptr}-\eqref{pw8mde0} has a unique  solution on $\R^2\times (0,T_*]$. First, it follows from \eqref{pw6-ict1}, \eqref{pw8mde0} and \eqref{pw7-ass1} that
\begin{equation*}
\begin{split}
A_1(0)\le \theta_0, \ A_2(0)=0,\ \norm{{\bf v}^\delta_0}_{L^{p_0}}\le M.
\end{split}
\end{equation*}
By \eqref{pw7-v0-lp}, we have
\begin{equation*}
\begin{split}
A_3(0)=\norm{{\bf v}^\delta_0}_{L^4}^4\le C(M)\theta_0^{4\eta_0}\le \theta_0^{\eta_0},
\end{split}
\end{equation*}
provided that $C(M)\theta_0^{3\eta_0}\le 1$.
Therefore, there exists a $T_1\in(0, T_*]$ such that \eqref{pw7-assup} holds for $T=T_1$.

Now, we set
\begin{equation}\label{pw7-t1}
\begin{split}
T^*=\sup\{T|\  \eqref{pw7-assup} \ \text{holds}\}.
\end{split}
\end{equation}
Then, $T^*>T_1>0$. Next, we claim that
\begin{equation}\label{pw7-tt}
\begin{split}
T^*=\infty.
\end{split}
\end{equation}
Otherwise, if $T^*<\infty$, by the blowup criterion \eqref{pw7-blowup} with $d=2$ and $q=4$, we
have that if
\begin{equation}\label{pw7-blowup1}
\begin{split}
\int_0^{T^*}\norm{{\bf v}^\delta}_{L^4}^4<\infty,
\end{split}
\end{equation}
then the solution can be extended beyond $T^*$. From Lemmas \ref{pw7-th1}-\ref{pw7-v66} and \eqref{pw7-t1}, we see that \eqref{pw7-assupf} holds for $T=T^*$, which immediately implies \eqref{pw7-blowup1}.
This means there exists a $T^{**}>T^*$ such that $(u-1, {\bf v})\in L^\infty([0,T^{**}], H^3)$ and \eqref{pw7-assup} holds for $T=T^{**}$, which
contradicts \eqref{pw7-t1}. Hence \eqref{pw7-tt} holds. Using \eqref{pw7-tt} and the blowup criterion \eqref{pw7-blowup} again, we complete the proof of Proposition \ref{pw7-global}.
\end{proof}
Note that Proposition \ref{pw7-global} holds for any given positive constant $\delta$ ($T_*, T_1$ and $T^*$ may depend on $\delta$). We cannot obtain the
$\delta$-independent estimates of the global solution $(u^\delta-1,\  {\bf v}^\delta)$ in $H^3$-norm.
However we can pass to the limit as $\delta\to 0$ in proper functional space in which the $\delta$-independent estimates of the solution are obtained, as shown below.

From \eqref{pw7-assupf}, we have the uniform-in-$\delta$ estimates as following:
\begin{equation}\label{pw7-as}
\begin{cases}\begin{split}
&\sup\limits_{t\in[0,T]}\left(\norm{u^\delta-1}^2+\norm{{\bf v}^\delta}^2\right)+\int_0^T\norm{\nabla u^\delta}^2 dt\leq \frac{3\theta_0}{2},\\
&\sup\limits_{t\in[0,T]}\left(\sigma\norm{\nabla u^\delta}^2+\sigma^2\norm{u^\delta_t}^2+\sigma^2\norm{{\bf v}^\delta_t}^2\right)+\int_0^T\sigma\norm{u^\delta_t}^2dt+\int_0^T\sigma^2\norm{\nabla u^\delta_t}^2dt
\le  \theta_0^\frac12,\\
&\sup\limits_{t\in[0,T]}\|{\bf v}^\delta\|_{L^4}^{4}+\int_{0}^t\|{\bf v}^\delta\|_{L^4}^{4}dt\le \theta_0^{\eta_0},\\
&\sup_{t\in[0,T]}\|{\bf v}^\delta\|_{L^{p_0}}\le 3M.
\end{split}\end{cases}
\end{equation}
On the other hand, by \eqref{G-Ns}, \eqref{pw7-as} and \eqref{pw7-u42}, one has
\begin{equation}\label{pw7-lu42}
\begin{split}
\sigma\norm{{u}^\delta-1}_{L^4}^2\le C\norm{{u}^\delta-1}\sigma\norm{\nabla{u}^\delta}\le C, \ \sigma^{2}\norm{\nabla{u}^\delta}_{L^4}^2\le C.
\end{split}
\end{equation}
Noticing that $\sigma=\min\{1, t\}$, we have from \eqref{pw7-as} and \eqref{pw7-lu42} that
\begin{equation}\label{1-le-last}
\begin{cases}
u^\delta-1\in L^\infty([0,\infty), L^2(\mathbb{R}^2)), \ \ \ \ \ (v^\delta_t, \nabla{u}^\delta)\in L^2([0,\infty), L^2(\mathbb{R}^2)), \\
u^\delta-1\in L^\infty((0,\infty),W^{1,4}(\mathbb{R}^2)),\ \ \ u^\delta_t\in L^2((0,\infty),H^1(\mathbb{R}^2)),\\
{\bf v}^\delta\in L^\infty([0,\infty); L^2(\mathbb{R}^2)\cap L^{p_0}(\mathbb{R}^2))\cap L^4([0,\infty); L^4(\mathbb{R}^2)).
\end{cases}
\end{equation}
By \eqref{1-le-last} and the Aubin-Lions-Simon lemma (cf.\cite{Rou}), we can extract a subsequence, still denoted by $(u^\delta, {\bf v}^\delta)$, such that the following convergence hold as $\delta\to 0$
\begin{equation*}
\begin{cases}
{\bf v}^\delta(\cdot,t)\to {\bf v}\ \text{strongly  in} \ C([0,\infty), H^{-1}(\mathbb{R}^2)), \\
u^\delta(\cdot,t)\to u(\cdot,t)\ \text{strongly  in} \ C((0,\infty), C(\mathbb{R}^2)),\\
\nabla u^\delta(\cdot,t)\to \nabla u(\cdot,t)\ \text{weakly in} \ L^2([0,\infty),L^2(\mathbb{R}^2)).
\end{cases}
\end{equation*}
Thus, the limit function $(u, {\bf v})$  is indeed a weak solution of the system  \eqref{hp}-\eqref{Boundary} and inherits all the bounds of \eqref{pw7-as} which yield \eqref{excon} and
\begin{equation}\label{pw7-ase}
\begin{cases}\begin{split}
&\sup\limits_{t\in[0,T]}\left(\norm{u-1}^2+\norm{{\bf v}}^2\right)+\int_0^T\norm{\nabla u}^2 dt\leq \frac{3\theta_0}{2},\\
&\sup\limits_{t\in[0,T]}\left(\sigma\norm{\nabla u}^2+\sigma^2\norm{u_t}^2+\sigma^2\norm{{\bf v}_t}^2\right)+\int_0^T\sigma\norm{u_t}^2dt+\int_0^T\sigma^2\norm{\nabla u_t}^2dt
\le  \theta_0^\frac12,\\
&\sup\limits_{t\in[0,T]}\|{\bf v}\|_{L^4}^{4}+\int_{0}^t\|{\bf v}\|_{L^4}^{4}dt\le \theta_0^{\eta_0},\\
&\sup_{t\in[0,T]}\|{\bf v}\|_{L^{p_0}}\le 3M.
\end{split}\end{cases}
\end{equation}
To complete the proof of  Theorem $\ref{pw7-th}$,
we only need to prove $\eqref{pw7-linff}$. It follows from \eqref{pw7-ase} and $\sigma=1$ for $t\geq1$ that
\begin{equation*}
\begin{split}
\int_{1}^\infty \norm{\nabla\tilde{u}}^2dt\le C
\end{split}\end{equation*}
and
\begin{equation*}
\begin{split}
\int_{1}^\infty\abs{\left(\norm{\nabla\tilde{u}}^2\right)_t}dt
\le& 2 \int_{1}^\infty \norm{\nabla\tilde{u}}\norm{\nabla\tilde{u}_t}dt\\
\le&\left(\int_{1}^\infty \norm{\nabla\tilde{u}}^2dt\right)^\frac12\left(\int_{1}^\infty \sigma^2\norm{\nabla\tilde{u}_t}^2dt\right)^\frac12 \le C.
\end{split}\end{equation*}
Thus,
\begin{equation}\label{pw7-u-in}
\norm{\nabla\tilde{u}}\to 0
~~\text{as}~~t\to\infty.
\end{equation}
Using Gagliardo-Nirenberg inequality \eqref{G-Ns} and $\sigma=1$ for $t\geq1$, we have
\begin{equation*}
\begin{split}
\sup\limits_{t\ge1}\norm{\tilde{u}}^4_{L^\infty}\le& C\sup\limits_{t\ge1}\sigma^{\frac52}\norm{\tilde{u}}_{L^4}^2\norm{\nabla\tilde{u}}_{L^4}^2\le C\sup\limits_{t\ge1}\norm{\tilde{u}}_{L^2}\sigma^{\frac 12}\norm{\nabla\tilde{u}}_{L^2}\sigma^2\norm{\nabla\tilde{u}}_{L^4}^2,
\end{split}
\end{equation*}
which together with \eqref{pw7-u42}, \eqref{pw7-ase} and \eqref{pw7-u-in} gives
\begin{equation}\label{pw7-uinl}
\norm{u-1}_{L^\infty}\to 0
~~\text{as}~~t\to\infty.
\end{equation}
By the interpolation inequality, \eqref{pw7-ase} and \eqref{pw7-uinl}, for any $2<p_1\le \infty$, we have
\begin{equation*}
\norm{u-1}_{L^{p_1}}\to 0
~~\text{as}~~t\to\infty.
\end{equation*}
On the other hand, by \eqref{pw7-vc1}, we have from \eqref{pw7-F4-0} and \eqref{pw7-vcf} that
\begin{equation*}
\begin{split}
\int_{1}^\infty\abs{\left(\norm{{\bf v}}_{L^4}^4\right)_t}dt\le& C\int_{1}^\infty\int_{\R^2}|\tilde{u}+1|{\bf v}^{4}dxdt+C\int_{1}^\infty\int_{\R^2}|{\bf F}||{\bf v}|^{3}dxdt\\
\le& C\int_{1}^\infty\int_{\R^2}{\bf v}^{4}dxdt+C\int_{1}^\infty\int_{\R^2}|{\bf F}||{\bf v}|^{3}dxdt\\
\le& C\int_{1}^\infty\int_{\R^2}{\bf v}^{4}dxdt+C\int_{1}^\infty\int_{\R^2}|{\bf F}|^{4}dxdt\\
\le& C.
\end{split}\end{equation*}
Combining the above inequality with \eqref{pw7-vcf}, we have
\begin{equation*}
\norm{{\bf v}}_{L^4}\to 0
~~\text{as}~~t\to\infty,
\end{equation*}
which together with the interpolation inequality, \eqref{pw7-L^4} and \eqref{pw7-vp0f} implies
\begin{equation*}
\norm{{\bf v}}_{L^{p_2}}\to 0
~~\text{as}~~t\to\infty,\ \ 2<{p_2}<p_0.
\end{equation*}
Hence $\eqref{pw7-linff}$ is proved and the proof of Theorem $\ref{pw7-th}$ is thus completed.

\section{Proof of Theorem  \ref{th-dlor} }
In this section, we pass the results of the transformed chemotaxis model (\ref{hp}) to the original chemotaxis system (\ref{ks1}). Noticing that the transformed and pre-transformed systems have the same quantity $u$, we are left to prove the results for $c$ only.
We start with the proof of \eqref{cdecay}. From the second equation
of \eqref{ks1} and the Cole-Hopf transformation \eqref{ch}, we can derive that
\begin{equation}\label{5-1}
(\ln c)_t=-u
\end{equation}
which together with $c_0>0$ and $u\geq0$ gives
\begin{equation}\label{5-2}
0\leq c(x,t)\le c(x,0)=c_0.
\end{equation}
Integrating \eqref{5-1} over $[1, t)$, we have from \eqref{5-2} that
\begin{equation}\label{5-3}
\begin{split}
c(x,t)=&c(x,1)\exp\left(-(t-1)-\int_1^t (u-1)d\tau\right)\\
\le &c_0\exp\left(-(t-1)-\int_1^t (u-1)d\tau\right) .
\end{split}
\end{equation}
From \eqref{3.38}, we get
\begin{equation*}
\begin{split}
\int_1^t \norm{u-1}_{L^\infty}d\tau\le \frac14 (t-1),
\end{split}
\end{equation*}
which, along with \eqref{5-3} yields
\begin{equation*}
\begin{split}
\norm{c}_{L^\infty}\le Ce^{-\frac34 (t-1)}\le  Ce^{-\frac34 t}.
\end{split}
\end{equation*}
This completes the proof of Theorem  \ref{th-dlor}.

\bigbreak \noindent \textbf{Acknowledgement}.  H.Y. Peng acknowledges a financial support from AMSS-PolyU
JRI in the Hong Kong Polytechnic University where he was a postdoctoral fellow, and support from the National Natural Science Foundation
of China No. 11901115. Z.A. Wang was supported in part by
the Hong Kong RGC GRF grant No. PolyU 153031/17P and internal grant No. ZZHY from HKPU. C.J. Zhu was supported by the National Natural Science
Foundation of China No. 11331005 and 11771150.


\addcontentsline{toc}{section}{\\References}
\begin{thebibliography}{99}
\bibitem{Adams} R.A. Adams and J. J.F. Fournier, {\it Sobolev spaces}, Pure and Applied Mathematics. 140 (2nd ed.). Boston, Academic Press, 2003.
\bibitem{Adler} J. Adler,  Chemotaxis in bacteria, {\it Science}, 153(1966), 708-716.
\bibitem{Chae2} M. Chae, K. Choi, K. Kang, and J. Lee. Stability of planar traveling waves in a
Keller-Segel equation on an infinite strip domain, {\it  J. Differential Equations}, 265:237-279, 2018.
\bibitem{Chen-xu-zhang}
M.T. Chen, X.Y. Xu and J. W. Zhang, Global weak solutions of 3D compressible micropolar fluids with discontinuous initial data and vacuum, {\it Commun. Math. Sci.,} 13(2015), 225-247.
\bibitem{CPZ1} L. Corrias, B. Perthame, and H. Zaag. A chemotaxis model motivated by angiogenesis. {\it C.
R. Math. Acad. Sci. Paris}, 2:141-146, 2003.
\bibitem{CPZ3} L. Corrias, B. Perthame, and H. Zaag. Global solutions of some chemotaxis and angiogenesis
systems in high space dimensions. {\it Milan J. Math}., 72:1-29, 2004.
\bibitem{DL} C. Deng and T. Li, Well-posedness of a 3D parabolic-hyperbolic Keller-Segel system in the Sobolev space framework, {\it J. Differential Equations}, 257(2014), 1311-1332.
\bibitem{Fan-zhao} J. Fan and K. Zhao, Blow up criteria for a hyperbolic-parabolic system arising from chemotaxis.
{\it J. Math. Anal. Appl.,} 394(2012), 687-695.
\bibitem{Friedman-A} A. Friedman, Partial Differential Equations, {\it Holt, Rinehart Winston,} New York, (1969).
\bibitem{Rafael1} R. Granero-Belinch\'{o}n, On the fractional Fisher information with applications to a hyperbolic-parabolic system of chemotaxis, {\it J. Differential Equations}, 262(2017), 3250-3283.
\bibitem{Rafael2} R. Granero-Belinch\'{o}n, Global solutions for a hyperbolic-parabolic system ofchemotaxis, {\it J. Math. Anal. Appl.}, 449(2017), 872-883.
\bibitem{GXZZ}
J. Guo, J.X. Xiao, H.J. Zhao, and C.J. Zhu, Global solutions to a
hyperbolic-parabolic coupled system with large initial data, {\it
Acta Math. Sci. Ser. B Engl. Ed.,} 29 (2009), 629-641.
\bibitem{Hao}   C. Hao,  Global well-posedness for a multidimensional chemotaxis model in critical Besov spaces, {\it Z. Angew Math. Phys.}, 63 (2012),  825-834.

\bibitem {Hoff-jde} D. Hoff, Global solutions of the Navier-Stokes equations for multidimensional compressible
flow with discontinuous initial data. {\it J. Differential
Equations,} 120(1995), 215-254.

\bibitem {Hoff-1995}
D. Hoff, Strong convergence to global solutions for multidimensional flows of compressible,
viscous fluids with polytropic equations of state and discontinuous initial data, {\it Arch. Rational
Mech. Anal.,} 132(1995), 1-14.

\bibitem{Hoff-arma} D. Hoff, Discontinuous solutions of the Navier-Stokes equations for multidimensional flows of
heat-conducting fluids. {\it Arch. Rational Mech. Anal.,} 139(1997),
303-354.

\bibitem {Hoff-cpam}
D. Hoff, Dynamics of singularity surfaces for compres
sible, viscous flows in two space dimensions,
{\it Comm. Pure Appl. Math.,} 55(2002), 1365-1407.

\bibitem {Hoff-santos}
D. Hoff and M. M. Santos, Lagrangean structure and propagation of singularities in multidimensional
compressible flow, {\it Arch. Ration. Mech. Anal.,} 188(2008), 509-543.

\bibitem{Hou} Q.Q. Hou, C.J. Liu, Y.G. Wang and Z.A. Wang, Stability of boundary layers for a viscous hyperbolic system arising from chemotaxis: one dimensional case, {\it SIAM J. Math. Anal}., 50:3058-3091, 2018.

\bibitem{HWJMPA} Q.Q. Hou and Z.A. Wang, Convergence of boundary layers for the Keller-Segel system with singular sensitivity in the half-plane, {\it J. Math. Pures Appl.}, 130:251-287, 2019.

\bibitem{Hu-lin}
 X.P. Hu and F.H. Lin,  Global solutions of two-dimensional incompressible viscoelastic flows with discontinuous initial data. {\it Comm. Pure Appl. Math.,} 69(2016), 372-404.


\bibitem{jin13}
H.Y. Jin, J.Y. Li, and Z.A. Wang, Asymptotic stability of traveling waves of a chemotaxis model with singular sensitivity, {\it J.
Differential Equations,} 255 (2013), 193-219.

\bibitem{jiu13}
Q.S. Jiu, Y. Wang and Z.P. Xin, Global well-posedness of the Cauchy problem of two-dimensional compressible Navier-Stokes equations in weighted spaces, {\it J.
Differential Equations,} 255 (2013), 351-404.

\bibitem{Kalinin} Y.V. Kalinin, L. Jiang, Y. Tu, and M. Wu, Logarithmic sensing in {E}scherichia coli bacterial chemotaxis,  {\it Biophysical J.}, 96(2009), 2439-2448.
\bibitem{KS} E. F. Keller and L. A. Segel, Traveling bands of chemotactic bacteria: A theoretical analysis, {\it J. Theor. Biol.}, 26 (1971), 235-248.
\bibitem{KLLZ-JDE}
H.H. Kong, H.L. Li, C.C.  Liang and G.J. Zhang, Global existence and exponential stability for the compressible Navier-Stokes equations with discontinuous data, {\it J. Differential Equations},  263(2017), 4267-4323.
\bibitem{Levine97}
 H.A. Levine and B.D. Sleeman, A system of reaction diffusion equtions arising in the theory of reinforced random walks, {\it SIAM J. Appl. Math.,} 57(1997),
683-730.
\bibitem{LSN} H.A. Levine, B.D. Sleeman, and M. Nilsen-Hamilton,  A mathematical model for the roles of pericytes and macrophages in the initiation of angiogenesis. I. the role of protease inhibitors in preventing angiogenesis,  {\it Math. Biosci.}, 168(2000), 71-115.
\bibitem{Li111}
 D. Li, T. Li, and  K. Zhao, On a hyperbolic-parabolic system modeling chemotaxis, {\it Math. Models Methods Appl. Sci.,} 21(2011), 1631-1650.
 \bibitem{Li-pan-zhao}
D. Li, R. Pan and K. Zhao, Quantitative decay of a hybrid type chemotaxis model with
large data, {\it Nonlinearity,} 28(2015), 2181-2210.
\bibitem{Li-Zhao-JDE} H. Li and K. Zhao, Initial-boundary value problems for a system of hyperbolic balance laws arising from chemotaxis, {\it J. Differential Equations}, 258(2015), 302-308.



\bibitem{Lij13}
J.Y. Li, L.N. Wang, and K.J. Zhang, Asymptotic stability of a
composite wave of two traveling waves to a hyperbolic-parabolic
system modeling chemotaxis, {\it Math. Methods Appl. Sci.,}
 36(2013), 1862-1877.

\bibitem{Li112}
T. Li, R.H. Pan, and K. Zhao, Global dynamics of a chemotaxis model on
bounded domains with large data, {\it SIAM J. Appl. Math.,}
72(2012), 417-443.
\bibitem{Li09}
T. Li and Z.A. Wang, Nonlinear stability of traveling waves to a
hyperbolic-parabolic system modeling chemotaxis, {\it SIAM J. Appl.
Math.,} 70(2009), 1522-1541.
\bibitem{Li10}
T. Li and Z.A. Wang, Nonlinear stability of large amplitude viscous shock waves of a
generalized hyperbolic-parabolic system arising in chemotaxis,
{\it Math. Models Methods Appl. Sci.,} 20(2010), 1967-1998.

\bibitem{Lions2} P.L. Lions.  Mathematical Topics in Fluid Mechanics.
 Vol. II, Compressible Models. {\it Clarendon Press, Oxford,} (1998).

\bibitem{Martinez}
V. Martinez, Z.A. Wang and K. Zhao, Asymptotic and viscous stability of large-amplitude solutions of a hyperbolic system arising from biology,
{\it Indiana Univ. Math. J}., 67:1383-1424, 2018.

\bibitem{MASMOUDI}
N. Masmoudi, Nader Global existence of weak solutions to the FENE dumbbell model of polymeric flows, {\it Invent. Math.,} 191(2013), 427-500.


\bibitem{Mei-peng-wang}
M. Mei, H. Peng, Z.A. Wang, Asymptotic profile of a parabolic-hyperbolic system with boundary effect arising from tumor angiogenesis
{\it  J. Differential Equations,} 259(2015), 5168-5191.

\bibitem{Othmer97}
H.G. Othmer and  A. Stevens, Aggregation, blowup, and collapse: the ABC's of taxis in reinforced random walks, {\it SIAM J. Appl.
Math.,} 57 (1997), 1044-1081.

\bibitem{PW2019} H.Y. Peng and Z.A. Wang,
On a parabolic-hyperbolic chemotaxis system with discontinuous data: well-posedness, stability and regularity, {\it J. Differnential Equations},
https://doi.org/10.1016/j.jde.2019.10.025, 2019.


\bibitem{RWWZZ} L.G. Rebholz, D. Wang, Z.A. Wang, K. Zhao and C. Zerfas, Initial boundary value problems for a system of parabolic conservation laws arising from chemotaxis in multi-dimensions, {\it Disc. Cont. Dyn. Syst.}, 39:3789-3838, 2019.


\bibitem{Rou}
T. Roub\'{i}\v{c}ek, {\it Nonlinear Partial Differential Equations with Applications (2nd ed.)},  Basel: Birkh\"{a}user, 2013.


\bibitem{Wang-review} Z.A. Wang, Mathematics of traveling waves in chemotaxis,
{\it Disc. Cont. Dyn. Syst.-Series B}.,18(3): 601-641, 2013.

\bibitem{Wang08}
Z.A. Wang  and T. Hillen, Shock formation in a chemotaxis Model, {\it Math. Methods Appl. Sci.,} 31 (2008), 45-70.

\bibitem{WWZ} D. Wang, Z.A. Wang and K. Zhao, Cauchy problem of a system of parabolic conservation laws arising from a Keller-Segel type chemotaxis model in multi-dimensions, {\it Indiana Univ. Math. J}., accepted, 2018.

\bibitem{Wang-xiang-yu}
Z.A. Wang, Z. Xiang and P. Yu,
Asymptotic dynamics on a singular chemotaxis system modeling onset of tumor angiogenesis,
{\it J. Differential Equations,} 260(2016), 2225-2258.
%
%

\bibitem{zhang07}
M. Zhang and C.J. Zhu, Global existence of solutions to a
hyperbolic-parabolic system, {\it Proc. Amer. Math. Soc.,}  135
(2007), 1017-1027.
\bibitem{zhang-tan-sun}
Y. Zhang, Z. Tan, and M.B. Sun, Global existence and asymptotic behavior of smooth
solutions to a coupled hyperbolic-parabolic system, {\it Nonlinear Analysis: Real World Applications,}
14(2013), 465-482.

\end{thebibliography}
\end{document}